\documentclass{amsart}

\usepackage[dvipsnames]{xcolor}
\usepackage{tikz}
\usepackage{amsmath,bm,bbm,amsthm, amssymb}
\usepackage{mathtools}
\usepackage{fullpage}
\usepackage{color}
\usepackage{dsfont}
\usepackage{animate}
\usepackage{etoolbox}
\usepackage{comment}
\usepackage{pgf}

\theoremstyle{plain}
\newtheorem{theorem}{Theorem}
\newtheorem{proposition}[theorem]{Proposition}
\newtheorem{corollary}[theorem]{Corollary}

\newtheorem{remark}[theorem]{Remark}

\AtBeginDocument{
	\label{CorrectFirstPageLabel}
	
}

\usepackage{bm}

\newtheorem{thm}{Theorem}[section]
\newtheorem{cor}[thm]{Corollary}
\newtheorem{lem}[thm]{Lemma}

\newtheorem{prop}[thm]{Proposition}
\theoremstyle{definition}

\def\bepr{\begin{proposition}}
\def\enpr{\end{proposition}}
\def\bec{\begin{corollary}}
\def\enc{\end{corollary}}

\def\xx{\mathbf x}
\def\yy{\mathbf y}

\def\d{\mathrm d}
\def\dtv{d_{\ms{TV}}}
\def\R{\mathbb R}
\def\N{\mathbb N}
\def\P{\mathbb P}
\def\PP{\mathbb P}
\def\E{\mathbb E}

\def\g{\gamma}

\def\es{\emptyset}

\def\ms{\mathsf}
\def\sm{\setminus}
\def\om{\omega}

\newcommand{\nc}{\newcommand}
\nc{\diam}{\mathsf{diam}}
\nc{\I}{{\mathds 1}}
\nc{\one}{{\mathds 1}}
\nc{\bN}{{\mathbf N}}
\nc{\bM}{{\mathbf M}}
\nc{\cB}{{\mathcal B}}
\nc{\cM}{{\mathcal M}}
\nc{\X}{{\mathbb X}}
\nc{\Z}{{\mathbb Z}}
\nc{\BX}{{\mathbb X}}
\nc{\BY}{{\mathbb Y}}
\nc{\cX}{{\mathcal X}}
\nc{\cY}{{\mathcal Y}}
\nc{\cN}{{\mathcal N}}
\nc{\cF}{{\mathcal F}}
\nc{\tv}{d_{\mathrm TV}}
\nc{\BP}{\mathbb{P}}
\nc{\BE}{\mathbb{E}}
\nc{\BQ}{\mathbb{Q}}
\nc{\RR}{\mathbb{R}}
\nc{\SSS}{\mathbb{S}}
\nc{\ind}[1]{\I_{#1}\,}

\newcommand{\del}{\mathbf{Del}}
\newcommand{\PPP}[1]{\operatorname{\mathbb{P}}\left(\,#1\,\right)}
\newcommand{\conv}[2][n]{\underset{#1\rightarrow #2}{\longrightarrow}}
\newcommand{\EEE}[1]{\operatorname{\mathbb{E}}\left[\,#1\,\right]}
\newcommand{\eq}[2][n]{\underset{#1\rightarrow #2}{\sim}}

\newcommand{\EE}{\mathbb{E}}
\theoremstyle{definition}

\numberwithin{equation}{section}

\definecolor{OliveGreen}{cmyk}{0.64,0,0.95,0.40}

\begin{document} 
\author{Nicolas Chenavier}
\author{Moritz Otto}

\address[Nicolas Chenavier]{Universit\'e du Littoral C\^ote d'Opale, 50 rue F. Buisson, 62228 Calais, France}
\email{nicolas.chenavier@univ-littoral.fr}	
\address[Moritz Otto]{Department of Mathematics, Aarhus University, Ny Munkegade 118, 8000 Aarhus C, Denmark}
\email{otto@math.au.dk}

\renewcommand{\thefootnote}{\fnsymbol{footnote}}

\title{Compound Poisson process approximation\\
 under $\beta$-mixing and stabilization} 

\begin{abstract} 
\noindent 
We establish Poisson and compound Poisson approximations for stabilizing statistics of $\beta$-mixing point processes and give explicit rates of convergence. Our findings are based on a general estimate of the total variation distance of a stationary $\beta$-mixing process and its Palm version. As main contributions, this article (i) extends recent results on Poisson process approximation to non-Poisson/binomial input, (ii) gives concrete bounds for compound Poisson process approximation  in a Wasserstein distance and (iii) illustrates the applicability of the general result in an example on minimal angles in the stationary Poisson-Delaunay tessellation. The latter is among the first (nontrivial) situations in Stochastic Geometry, where compound Poisson approximation can be established with explicit extremal index and cluster size distribution.
\end{abstract}
\maketitle


\vspace{0.3cm}


\section{Introduction}
One of the most active fields in Stochastic Geometry concerns extremes of geometric features that are based on point processes. To modelize this, let $z:(\R^d)^k\rightarrow\R$ be a center function and let $g:(\R^d)^k\times \mathbf{N}\rightarrow\{0,1\}$, $k\geq 1$, be measurable and translation-invariant, where $\mathbf{N}$ denotes the set of locally finite couting measures. Given a point process $\xi$ in $\R^d$, a natural quantity is:
\[\Xi = \sum_{\mathbf{x}\in \xi_{\neq}^k}g(\mathbf{x},\xi)\delta_{z(\mathbf{x})}.\]
The point process $\Xi$ often arises to deal with extremes in Stochastic Geometry. For example, if $k=1$, $z(x)=x$, $g(x, \xi)=\one\{d(x,\xi\setminus\{x\})>v_n\}\one\{x\in W_n\}$, where $d(x,\xi\setminus\{x\})$ denotes the distance between $x$ and $\xi\setminus\{x\}$ and where $W_n=[-n^{1/d}/2,n^{1/d}/2]^d$ is an arbitrary large window, then $\Xi$ counts the points in  $W_n$ for which the distance to the nearest neighbor is larger than some (suitable) threshold $v_n$. To describe the point process $\Xi$, many Poisson approximation results have been recently established. For instance, in \cite{S05,S09}, Poisson approximations  for dependent thinnings of point processes that have a density with respect to a Poisson process are derived. In \cite{ST12,ST16}, the Malliavin calculus on the Poisson space is used  to determine scaling limits for Poisson $U$-statistics. In \cite{C14,CH16,CN19,PS22}, the Chen-Stein method is  applied to deal with the order statistics of various quantities for random tessellations. In \cite{CHO22}, Poisson approximation for the volume of $k$-th nearest neighbor balls is discussed in the Euclidean space and  in the hyperbolic space. Rates of convergence for the Poisson approximations are generally provided, in terms of total variation distance and Kantorovich-Rubinstein (KR) distance. To the best of our knowledge, thanks to a coupling approach, the most precise results for statistics of Poisson point processes can be found in \cite{BSY22,O22}. These results are formulated in the strong KR distance and have been generalized to Poisson hyperplanes in \cite{O21} and to Gibbs processes in \cite{LS2019}. However, Poisson process approximation in the KR distance requires a Palm coupling which is in general hard to establish. Therefore, in this paper we focus on (compound) Poisson approximation for the large class of $\beta$-mixing processes in a weaker Wasserstein distance.  

Poisson approximation in general requires two assumptions: a mixing condition for the point process $\xi$ and anti-clustering condition arising in $g(\mathbf{x},\xi)$. When the latter is not satisfied, the exceedances (namely the set of points  $z(\mathbf{x})$ such that $g(\mathbf{x},\xi)\neq 0$) can be realized into clusters. In this case, under suitable assumptions, the point process $\Xi$ can be approximated by a compound Poisson point process (see \cite{BMP22,CR18} for examples concerning random tessellations and the $k$-th nearest neighbor graph). However, to the best of our knowledge, no rate of convergence is established in the context of Stochastic Geometry. 

In this paper, we make explicit the rate of convergence for compound Poisson approximation. As first results, we establish  general upper bounds in terms of total variation distance and Wasserstein distance (Propositions \ref{prop:mix} and \ref{pr:cp_app}). Then we apply the latter to have precise estimates for stabilizing statistics, first for deriving Poisson approximation (Theorem \ref{Th:Poi}) then for deriving compound Poisson approximation (Theorem \ref{Th:CP}). Our results are based on Palm couplings. We illustrate our main theorems through two examples. The first one concerns the largest distances to the nearest neighbor when the underlying point process is the so-called Gauss-Poisson process and extends a result of  \cite{C14}. The second one is unprecedented and concerns the smallest angles in a Poisson-Delaunay tessellation. In particular, we make explicit the so-called extremal index and the cluster size distribution. Except a recent work due to Basrak \textit{et al.} \cite{BMP22} it is the first time it is done in the context of Stochastic Geometry.

Our paper is organized as follows. In Section \ref{sec:preliminaries}, we recall some notions on point processes and give two technical results. In Section \ref{sec:stabilizing}, we apply the latter to obtain Poisson and compound Poisson approximations with explicit rates of convergence when the function $g$ is exponentially stabilizing. Our theorems are illustrated through the two examples discussed above in Section \ref{sec:applications}.

\section{Preliminaries}
\label{sec:preliminaries}
In this section, we  recall some notions on point processes and  establish two technical results which will be used to derive our main theorems.

\subsection{Basic notions on point processes}

\subsubsection{Point processes}
\label{sec:basicscompound}
First, we give some notation. In what follows, we denote by $||\cdot ||$ the Euclidean norm of $\R^d$, $d\geq 1$, and by $\mathcal B^d$ the Borel $\sigma$-algebra. For any $B\in \mathcal{B}^d$, we denote by $|B|$ the volume of $B$. We write $\mathbf{N}$ for the space of all $\sigma$-finite counting measures on $\R^d$ and $\widehat{\mathbf{N}}$ for the space of all finite counting measures on $\R^d$. We equip $\mathbf{N}$ and $\widehat{\mathbf{N}}$ with the corresponding $\sigma$-algebras that are induced by the mappings $\omega \mapsto \omega(B)$ for all $B \in \mathcal{B}^d$. The latter are denoted $\mathcal{N}$ and $\widehat{\mathcal{N}}$ and are the Borel $\sigma$-algebras with respect to the Fell topologies on $\mathbf{N}$ and $\widehat{\mathbf{N}}$, respectively (see e.g. \cite{SW08}, p. 563).

A \textit{point process} in $\R^d$ is a random variable $\xi: \Omega\rightarrow \mathbf{N}$ defined on some probability space $(\Omega, \mathcal{A}, \PP)$. The point process is said \textit{simple} if $\xi(\{x\}) \in \{0,1\}$ for any $x\in \R^d$. In this case, with a slight abuse of notation, we identify the point process $\xi$ to its support, which is a random closed subset in $\R^d$. 


Now, let $\pi_1,\pi_2,\dots$ be finite measures on $\R^d$. A {\em compound Poisson point process} with parameters $\pi_1,\,\pi_2,\dots$ is a point process $\psi$ of the form
\begin{align*}
	\psi(B)=\sum_{i \ge 1} i\zeta(B \times \{i\}) ,\quad B \in \mathcal B^d,
\end{align*}
where $\zeta$ is a Poisson point process in $\R^d\times \mathbb{N}^*$ with intensity measure $\E [\zeta]$ given by $\E [\zeta (B\times \{i\})]=\pi_i(B)$ for $B \in \mathcal B^d$ and $i\in \mathbb{N}$. In what follows, we let $\pi:=\pi_1+\pi_2+\dots$ and write $\mathsf{CP}(\pi_1,\pi_2,\dots)$ for the distribution of a compound Poisson process with parameters $\pi_1,\pi_2,\dots$  and $\mathsf{Po}(\pi)$ for the distribution of a Poisson process with intensity measure $\pi$. Intuitively, $\pi_i$ is the intensity measure for clumps of size $i$. If $\psi$ is stationary, then there exists $\gamma>0$ (referred to as the \textit{intensity} of $\psi$) and a distribution $\mathbb{Q}$ (referred to as the \textit{cluster size distribution}) on $\mathbb{N}^*:=\{1,2,\dots\}$ such that $\pi_i(B) = \gamma |B| \mathbb{Q}(i)$ for any $i\geq 1$. In this case, $\psi$  can be identified to its support, which is a random set of points of the form $\{(x_k,m_k):k\geq 1\}\subset \R^d\times \mathbb{N}^*$,  where  $\{x_k:k\geq 1\}$ is a stationary Poisson point process of intensity $\gamma$ in $\R^d$ and where the $m_k$'s are i.i.d., with distribution $\mathbb{Q}$, and independent of the $x_k$'s. We have the  
relation \[\zeta(B\times \{i\}) = \sum_{k\geq 1}\one\{x_k\in B\}\one\{m_k=i\}.\]

\subsubsection{Beta-mixing and exponential decay of dependence}
Let $\xi$ be a point process in $\mathbb{R}^d$. To quantify the spatial decay of dependence, we introduce the notion of \textit{$\beta$-mixing coefficient}. The latter is defined as follows (see \cite{CX22}): for any $A, B \in \mathcal B^d$,
\[\beta_{A,B}^{\xi}:=\frac 12  \int_{{\mu}_1\in \mathcal{N}_A}\int_{{\mu}_2 \in \mathcal{N}_B} \left|\PPP{\xi_{A} \in \mathrm{d}\mu_1,\xi_{B} \in \mathrm{d} \mu_2}-\PPP{\xi_A\in \mathrm{d}\mu_1} \PPP{\xi_B \in \mathrm{d}\mu_2}\right|,\]
where $\xi_{C}:=\xi \cap C$ denotes the restriction of $\xi$ to $C$ and $\mathcal N_C$ is the induced $\sigma$-field of $\mathcal N$ on $C$, for any $C\in \mathcal B^d$. In what follows, we let $\diam(A)=\sup\{||x-y||, x, y\in A\}$  and $d(A,B)=\inf\{||x-y||, x\in A, y\in B \}$ for any $A,B\in \mathcal B^d$. 

In the sense of \cite{CX22}, we say that the point process $\xi$ has the {\em exponential decay of dependence} ($\mathsf{EDD}$) property if there are constants $\theta_0 \in [0,+\infty)$ and $\theta_i \in (0,+\infty),\, 1 \le i \le 4,$ such that for all $A,B \in \mathcal B^d$ with $d(A,B)\ge \theta_3 \log(\diam(A) \vee \diam(B) \vee \theta_4)$,
\begin{align*}
	\beta^\xi_{A,B} \le \theta_1 (\mathrm{diam}(A)^{\theta_0} \vee 1)(\mathrm{diam}(B)^{\theta_0} \vee 1) e^{-\theta_2d(A,B)}.
\end{align*}
Roughly speaking, the $\mathsf{EDD}$ property says that the total variation distance between the law of $(\xi_A, \xi_B )$ and the law of the independent union of $\xi_A$ and $\xi_B$ decays exponentially fast as the
distance between $A$ and $B$ becomes large.

There is a wide class of point processes that satisfy the $\mathsf{EDD}$ property, including Gibbs point processes with nearly finite range potentials, determinantal point processes with fast decay kernels (in the sense of \cite[Lemma 3.2]{CX22}), $r$-dependent point processes and Boolean models (see Section 3 in \cite{CX22}). 


\subsubsection{Palm processes}
We briefly recapture some basic facts from Palm theory that we need to formulate and prove our results. Let $\xi$ be a point process in $\R^d$. We call $\xi^{x}$ a \textit{Palm version} of $\xi$ at the point $x \in \R^d$, if for any measurable function $f:\R^d \times \mathbf N \to \R_+:=[0,\infty)$,
\begin{equation}
\label{eq:defpalm}
\EEE{\int_{\R^d} f(x,\xi) \xi(\mathrm d x)}= \EEE{\int_{\R^d} f(x,\xi^{x}) \E[\xi](\mathrm d x)},
\end{equation}
where  $\E[\xi]$ denotes the \textit{intensity measure} defined by $\E[\xi](A)=\EEE{\xi(A)}$ for any Borel subset $A$ in $\R^d$.

More generally, we will need to consider the Palm version of $\xi$ with respect to a subprocess. To do it, let $k \in \N^*$ and $h:(\R^d)^k\times \mathbf N \to \{0,1\}$, $z:(\R^d)^k\rightarrow \R^d$  be two measurable functions such that $h$ is translation-invariant and $z$ is translation-covariant, which means that $z(x_1+y,\dots,x_k+y)=z(x_1,\dots,x_k)+y$ for any $x_1,\dots,x_k, y \in \R^d$. Let $D \subset \R^d$ be a compact subset in $\R^d$. Define a subprocess $\zeta$ of $\xi$ by
\begin{equation}
	\zeta:=\sum_{\xx \in \xi_{\neq}^k} h\left(\xx,\xi_{D_{z(\xx)}}\right)  \delta_{z(\xx)}, \label{def:Xil}
\end{equation}
where $D_z:=D+z$, $z\in \R^d$. We call $\xi^{z,\zeta}$ a Palm version of $\xi$ with respect to $\zeta$ if, for any measurable function $f:\R^d \times \mathbf N \to \R_+:=[0,\infty)$,
$$
\EEE{\int_{\R^d} f(x,\xi) \zeta(\d x)}= \EEE{\int_{\R^d} f(x,\xi^{x,\zeta}) \E[\zeta](\d x)}. 
$$
Intuitively, $\xi^{x,\zeta}$ describes the process $\xi$ given that $x$ is an atom of the subprocess $\zeta$.

\subsubsection{Distances of point processes}
We introduce below some distances following the same notation as \cite{BM02}. For any $x,y \in \R^d$, let $d_0(x,y):=||y-x|| \wedge 1$.   Let $\mathcal K$ be the set of all functions $k:\R^d \to \R$ such that
$$
\sup_{x \neq y \in \R^d} |k(x)-k(y)|/d_0(x,y)\le 1,
$$
and define a distance $d_1$ between finite measures $\mu$ and $\chi$ by

\begin{align*}
	d_1(\mu,\chi):&=
	\begin{cases}
		1,\quad &\mu(\R^d)\neq \chi(\R^d),\\
		0,\quad &\mu(\R^d)=\chi(\R^d)=0,\\
		m^{-1} \sup_{k \in \mathcal K} \left|\int k\, \d \mu - \int k\, \d \chi\right|,\quad &\mu(\R^d)=\chi(\R^d)=m>0.
	\end{cases}
\end{align*}
If $\mu=\sum_{i=1}^m \delta_{x_i}$ and $\chi=\sum_{i=1}^{m} \delta_{y_i}$ are finite counting measures on $\R^d$, $d_1$ is alternatively given by
$$
d_1(\mu,\chi)=\min\limits_{\pi \in S_m} \left\{m^{-1} \sum\limits_{i=1}^m d_0\left(x_i,y_{\pi(i)}\right)\right\}.
$$
The metric $d_1$ is sometimes called the {\em Kantorovich$_{1,1}$}- or {\em Wasserstein-metric induced by $d_0$} in the literature, see e.g. \cite{BSY22}.

We define the distance $\mathbf{d_2}$ between two finite point processes $\xi$ and $\nu$ in $\R^d$ by
\begin{align*}
	\mathbf{d_2}(\xi,\nu):=\sup\limits_{f \in \mathcal F}  \left| \EEE{f(\xi)}-\EEE{f(\nu)}\right|,
\end{align*}
where $\mathcal F$ is the class of all 1-Lipschitz functions with respect to $d_1$. By the Kantorovich-Rubinstein duality (see \cite[Theorem 5.10]{Vil08}), we have 
\begin{equation}
\label{eq:KRduality}
	\mathbf{d_2}(\xi,\nu)=\inf\limits_{C \in \Sigma(\xi,\nu)} \int_{\bN \times \bN} d_1\left(\mu_1,\mu_2\right)\,C(\mathrm{d}(\mu_1,\mu_2)),
\end{equation}
where $\Sigma(\xi,\nu)$ is the set of all probability measures on $\bN\times \bN$  with marginal distributions $\P(\xi\in \cdot)$ and $\P(\nu\in \cdot)$.

For finite measures $\mu, \chi$ on $\R^d$ we also write
\begin{align*}
	\widehat{d}_1(\mu,\chi):=
	\begin{cases}
		\inf\limits_{\mu' \le \mu:\,\mu'(\R^d)=\chi(\R^d)}d_1(\mu',\chi),\quad &\mu(\R^d)\ge\chi(\R^d),\\
		\widehat{d}_1(\chi,\mu),\quad &\mu(\R^d)<\chi(\R^d).
	\end{cases}
\end{align*}

\subsection{Technical results}

\subsubsection{A bound for the total variation distance }
Let $A$ be a subset in $\R^d$ such that $A^c$ is compact and contains the origin. In what follows, we write $A_z=A+z$ for any $z \in \R^d$  and assume that the mapping 
\begin{align}
	\R^d \times \mathbf N \to \R^d \times \mathbf N:(z,\omega) \mapsto (z,\omega_{A_z}) \label{def:Ax}
\end{align}
is product measurable. In our applications, the set $A$ will be the complement of a ball centered at the origin. If so, \cite[page 15]{BB92} shows that the mapping in \eqref{def:Ax} is product measurable. 

\begin{lem}\label{lem:mix}
	Let $\xi$ be a simple point process in $\R^d$, let $A,B \subset \R^d$ be compact sets and let $F:\mathbf{N} \to [0,+\infty)$ be a measurable function. For some compact $D \subset \R^d$ and a measurable function $h: (\R^d)^k\times \mathbf N \to \{0,1\}$, let $\zeta$ be the process given at \eqref{def:Xil}.  Then we have for any $p>0,\,q>0$ with $1/p+1/q\le 1$,
	\begin{align*}
		\left|	\int_B	\E	[F(\xi_{A_z}^{z, \zeta})-F(\xi_{A_z})]\,\E[\zeta](\d z)\right| \le  4 \left(\E [\zeta(B)^p] \right)^{\frac 1p} \ \left(\EEE{\sup_{z \in B} F(\xi_{A_z})^q} \right)^{\frac 1q} (\beta^{\xi}_{B+D, B+A})^{1-\frac{1}{p}-\frac 1q}.
	\end{align*}
\end{lem}
Given two point processes $\xi_1, \xi_2$, denote the total variation distance between $\xi_1$ and $\xi_2$ by \[\dtv(\xi_1, \xi_2) = \sup_{\mathbf{A}\in \mathcal{N}}|\PPP{\xi_1\in \mathbf{A}} - \PPP{\xi_2\in \mathbf{A}}|.\] As a direct consequence of the above lemma, we derive  an upper bound for the total variation distance between $\xi$ and its Palm process. 
\begin{prop}
\label{prop:mix}
With the same notation as in Lemma \ref{lem:mix}, if $\xi$ and $\zeta$ are stationary, then for any $p>0,\,q>0$ with $1/p+1/q\le 1$, 

\[
\dtv(\xi_{A}, (\xi^{o,\zeta})_{A})  \le 4 \frac{(\E [\zeta(B)^p])^{\frac 1p}}{\E [\zeta](B)} (\beta^{\xi}_{B+D, B+A})^{1-\frac{1}{p}-\frac 1q}.
\]
\end{prop}
The above inequality generalizes \cite[Theorem 1]{HP13} in the sense that the Palm version of $\xi$ is with respect to $\zeta$ and not with respect to $\xi$. In particular, if $\xi$ is stationary and $D=\{o\}$, we have
\begin{equation}
	\label{rem:mix}
	\dtv(\xi_{A}, (\xi^{o})_{A})  \le 4 \frac{(\E [\xi(B)^p])^{\frac 1p}}{\E [\xi](B)}(\beta^{\xi}_{B, B+A})^{1-\frac{1}{p}-\frac 1q}.
\end{equation}

\begin{proof}[Proof of Lemma \ref{lem:mix}]
	By definition of $\xi^{z,\zeta}$, we have
	\begin{align*}
		\int_B	\E	[F(\xi_{A_z}^{z,\zeta})] \,\E[\zeta](\d z)=\E \Big[\sum_{z\in \zeta} \one\{z \in B\} F( \xi_{A_z})\Big]=\E \Big[\sum_{\xx \in \xi_{\neq}^k} \one\{z(\xx) \in B\}  h(\xx,\xi_{D_{z(\xx)}}) F(\xi_{A_{z(\xx)}})\Big].
	\end{align*} 
	In the same way, we have
	\begin{align*}
		\int_B	\E	[F(\xi_{A_z})] \,\E[\zeta](\d z)=\E \Big[ \sum_{\xx \in \xi_{\neq}^k} \one\{z(\xx) \in B\}  h(\xx,\xi_{D_{z(\xx)}}) \E[F(\xi_{A_{z(\xx)}})] \Big].
	\end{align*} 
	Now, let $\tilde \xi_{B + D}\stackrel{d}{=}\xi_{B+D}$ and $\tilde \xi_{A + B}\stackrel{d}{=}\xi_{A+B}$ be independent point processes. We couple $(\xi_{B+D},\xi_{B+A})$ and $(\tilde \xi_{B+D},\tilde \xi_{B+A})$ in such a way that
	\begin{align}
		\P((\xi_{B+D},\xi_{B+A})\neq (\tilde \xi_{B+D},\tilde \xi_{B+A})) = \dtv((\xi_{B+D},\xi_{B+A}),(\tilde \xi_{B+D},\tilde \xi_{B+A})) =:2\beta_{B+D,B+A}^{\xi}. \label{eqn:coupxi}
	\end{align}
	This is possible due to the coupling lemma \cite[p. 254]{BHJ92}. Because $A_{z(\mathbf{x})}=(A+B)\cap A_{z(\mathbf{x})}$ and  $D_{z(\mathbf{x})}=(B+D)\cap D_{z(\mathbf{x})}$ provided that $z(\mathbf{x})\in B$, we get
\[
		\int_B	\E	[F(\xi_{A_x}^{x,\zeta})-F(\xi_{A_x})]\, \E [\zeta](\d x)=\E [H(\xi_{B + D}, \xi_{A + B})] - \E [H(\tilde \xi_{B + D}, \tilde \xi_{A + B})],
\]
	where $H(\mu,\nu):=\sum_{\mathbf{x} \in \mu_{\neq}^k} \one\{z(\mathbf{x}) \in B\} h\left(\mu_{D_{z(\mathbf{x})}}\right)F\left(\nu_{A_{z(\mathbf{x})}}\right)$. The following argument is adopted from \cite[Lemma 1]{Yos76} to our setting. Let $\delta>0$, $\gamma:=\frac{1}{1+\delta}$ and define
	\begin{align*}
		M:=&\max\left\{\E[H(\xi_{B + D}, \xi_{B + A})^{1+\delta}], \E[H(\tilde \xi_{B + D}, \tilde \xi_{B + A})^{1+\delta}] \right\},\\
		E:=&\left\{(\mu,\nu)\in \mathbf N^2:\, H(\mu,\nu)\le M^{\gamma} (\beta^{\xi}_{B+D,B+A})^{-\gamma}\right\}.
	\end{align*}
	Then we have by the definition of $E$ and \eqref{eqn:coupxi},
	\begin{multline}
		\label{eq:BH}
		\Big| \EEE{H(\xi_{B+D},\xi_{B+A}) \one\{(\xi_{B+D},\xi_{B+A}) \in E\}- H(\tilde \xi_{B+D},\tilde \xi_{B+A}) \one\{(\tilde \xi_{B+D},\tilde \xi_{B+A}) \in E\}} \Big|\\
		\begin{split}
			& \le  M^{\gamma} (\beta^{\xi}_{B+D, B+A})^{-\gamma}\PPP{(\xi_{B+D},\xi_{B+A})\neq (\tilde \xi_{B+D},\tilde \xi_{B+A})}\\
			&  = 2 M^{\gamma} (\beta^{\xi}_{B+D, B+A})^{ \gamma \delta}. 
		\end{split}
	\end{multline}
	
	Moreover, because $1< M^{-\gamma\delta} (\beta^{\xi}_{B+D, B+A})^{\gamma\delta }  H(\xi_{B+D},\xi_{B+A})^{\delta}  $ on the event $E$, we have
	\begin{multline}
		\EEE{H(\xi_{B+D},\xi_{B+A})\one\{(\xi_{B+D},\xi_{B+A}) \notin E\}}\\
		\le  M^{-\gamma\delta} (\beta^{\xi}_{B+D, B+A})^{\gamma\delta } 	\EEE{ H(\xi_{B+D},\xi_{B+A})^{1+\delta}\one\{(\xi_{B+D},\xi_{B+A}) \notin E\}}\le M^{1-\gamma\delta} (\beta^{\xi}_{B+D, B+A})^{\gamma\delta},\label{eq:BcH}
	\end{multline}
	and analogously,
	\begin{align}
		&\EEE{H(\tilde \xi_{B+D},\tilde \xi_{B+A})\one\{(\tilde \xi_{B+D},\tilde \xi_{B+A}) \notin E\}}\le M^{1-\gamma\delta} (\beta^{\xi}_{B+D, B+A})^{\gamma\delta},\label{eq:BcH2}
	\end{align}
	Combining \eqref{eq:BH},  \eqref{eq:BcH}, \eqref{eq:BcH2} and using the fact that $1-\gamma\delta = \gamma = \frac{1}{1+\delta}$, we have
	\begin{multline}
		\Big|\int_B	\E	[F(\xi_{A_x}^{x, \zeta})-F(\xi_{A_x})]\, \E[\zeta](\d x)\Big|\\
		\le 4 \max \{\E [H(\xi_{B+ D}, \xi_{B+A})^{1+\delta}],\E [H(\tilde \xi_{B+D},\tilde \xi_{ B+A})^{1+\delta}] \}^{\frac{1}{1+\delta}} (\beta^{\xi}_{B+D, B+A})^{1-\frac{1}{1+\delta}}. \label{eq:palmmax}
	\end{multline}
	From Hölder's inequality we find for any $p,q \in [1,+\infty]$ with $1/p+1/q=1/(1+\delta)$, 
	\begin{align*}
		(\E [H(\xi_{B + D}, \xi_{B+A})^{1+\delta}])^{\frac{1}{1+\delta}}&\le  \Big(\E \Big[ (\zeta(B))^{1+\delta} \big(\sup_{x \in B} F(\xi_{A_x})\big)^{1+\delta}\Big]\Big)^{\frac{1}{1+\delta}}\\
		&\le \Big(\E [\zeta(B)^p] \Big)^{\frac 1p} \ \Big(\E \Big[\sup_{x \in B} F(\xi_{A_x})^q\Big] \Big)^{\frac 1q}.
	\end{align*}
	Since an analogous bound holds for $(\E |H(\tilde \xi_{B + D}, \tilde \xi_{B+A})|^{1+\delta})^{\frac{1}{1+\delta}}$ as well, we deduce from \eqref{eq:palmmax} that  \[\Big|\int_B	\E	[F(\xi_{A_x}^{x, \zeta})-F(\xi_{A_x})]\, \E[\zeta](\d x)\Big|
	\leq  4 \Big(\E [\zeta(B)^p] \Big)^{\frac 1p} \ \Big(\E \Big[\sup_{x \in B} F(\xi_{A_x})^q\Big] \Big)^{\frac 1q} (\beta^{\xi}_{B+D, B+A})^{1-\frac{1}{p}-\frac 1q}.\]
\end{proof}

\subsubsection{A general upper bound for compound Poisson process approximation}

Now, let $\Xi$ be a simple point process in $\R^d$ and let $C$ be a compact subset in $\R^d$. For any $i\in \N$, we consider below the subprocess
\begin{align*}
	\Xi^{(i)}:=\sum_{z \in \Xi} \one\{\Xi(C_z)=i\} \delta_{z}.
\end{align*}

The following proposition provides an upper bound between $\Xi$ and a compound Poisson point process.

\begin{prop}
	\label{pr:cp_app}
	Let $\Xi$ be a simple finite point process such that $\E[\Xi(\R^d)]<\infty$. Let $A\subset \R^d$ such that $A^c$ is compact, contains the origin and satisfies the measurability condition \eqref{def:Ax}. Let $\mathsf{CP}(\pi_1,\pi_2,\dots)$ be the distribution of a compound Poisson process with parameters $\pi_1$, $\pi_2$, \ldots  Then 
	\begin{multline*}
		\mathbf{d_2}(\mathcal L \Xi, \mathsf{CP}(\pi_1,\pi_2,\dots))\le e^{\E[\Xi(\R^d)]} \Big(\sum_{i \ge 1} \Big\{ |i\pi_{i}(\R^d)- \E[\Xi^{(i)}(\R^d)] | + \min\{\pi_{i}(\R^d), \E[\Xi^{(i)}]/i\} \hat d_1(\pi_{i}, \E[\Xi^{(i)}]/i) \Big\}\\
		+\Big\{ \int_{\R^d} \big( {\sup_{y \in A^c} (d_0(y,z))}+2\E[\Xi(A_z^c) ] \big)\, \E[\Xi](\d z)+2\sum_{i \ge 1} \int_{\R^d} \dtv((\Xi^{zi})_{A_z},\Xi_{A_z}) \, \E[\Xi^{(i)}](\d z)\Big\}\Big).
	\end{multline*}
\end{prop}

\begin{proof}[Proof of Proposition \ref{pr:cp_app}]
	Let
	\begin{align*}
		i \pi'_i(\d z) &:=\E[\Xi^{(i)}](\d z).
	\end{align*} 
and $\pi':=\sum_{i \ge 1} \pi_i'$. We first bound $\mathbf{d_2}(\mathcal L \Xi, \mathsf{CP}(\pi'_1,\pi'_2,\dots))$. To do it, let us consider a coupling $(\tilde \Xi^{zi}, \tilde \Xi)$ of $(\Xi^{zi}, \Xi)$ such that $\P(\tilde \Xi^{zi} \neq \tilde \Xi)=\dtv(\Xi^{zi},\Xi)$, where $\Xi^{zi}$ is the Palm version of $\Xi^{(i)}$. According to \cite[Theorem 4.2]{BM02}, we have 
	\begin{multline*}
		\mathbf{d_2}(\mathcal L \Xi, \mathsf{CP}(\pi'_1,\pi'_2,\dots))\le e^{\pi'(\R^d)} \Big\{ \int_{\R^d} \E[d_1((\Xi^z)_{A_z^c}, \Xi^z(A_z^c)\delta_z)]\, \E[\Xi](\d z)\\
		+\sum_{i \ge 1}  \int_{\R^d} \left(\P(\tilde \Xi^{zi}(A_z)\neq \tilde \Xi (\R^d) ) + \E [\hat d_1( (\tilde \Xi^{zi})_{A_z},\tilde \Xi )]\right) \,\E[\Xi^{(i)}](\d z)\Big\}.
		\end{multline*}
	 	Here, we use the bound $d_1((\Xi^z)_{A_z^c}, \Xi^z(A_z^c)\delta_z)\le {\sup_{y \in A_z^c} d_0(z,y)}$ a.s. Moreover, 
	\begin{equation*}
		\P(\tilde \Xi^{zi}(A_z)\neq \tilde \Xi (\R^d)) \le \P(\tilde \Xi (A_z^c)>0) + \P(\tilde \Xi^{zi}(A_z)\neq \tilde \Xi(A_z)) \le \E[\Xi(A_z^c)] + \dtv((\Xi^{zi})_{A_z},\Xi_{A_z}).
	\end{equation*}
	The definition of $\hat d_1$ gives that
	$
	\hat d_1((\tilde \Xi^{zi})_{A_z}, \tilde \Xi) \le \one\{\tilde \Xi_{A_z^c} \neq \es\} + \one\{(\tilde \Xi^{zi})_{A_z} \neq \tilde \Xi_{A_z}\},
	$ 
	from which we get that
\begin{equation*}
	\E \big[\hat d_1((\tilde \Xi^{zi})_{A_z}, \tilde \Xi)\big]  \le \P(\tilde \Xi_{A_z^c} \neq \es) + \P((\tilde \Xi^{zi})_{A_z} \neq \tilde \Xi_{A_z}) \leq \EEE{\Xi(A_z^c)} + \dtv((\Xi^{zi})_{A_z},\Xi_{A_z}). 
\end{equation*}
	Therefore, 
	\begin{multline*}
		\d_2(\mathcal L \Xi, \mathsf{CP}(\pi'_1,\pi'_2,\dots))\le  e^{\pi'(\R^d)} \Big\{ \int_{\R^d} \big( \sup_{y \in A_z^c} d_0(z,y)+2\E[\Xi(A_z^c) ] \big)\, \E[\Xi](\d z)\\
		+2\sum_{i \ge 1} \int_{\R^d} \dtv((\Xi^{zi})_{A_z},\Xi_{A_z}) \, \E[\Xi^{(i)}](\d z)\Big\}.
	\end{multline*}
	
	To complete the argument, we use that by  \cite[Theorem 4.6]{BM02},  
	\[\mathbf{d_2}( \mathsf{CP}(\pi_1,\pi_2,\dots), \mathsf{CP}(\pi_1',\pi_2',\dots))
	\le e^{\pi'(\R^d)} \sum_{i \ge 1} \Big\{i |\pi_{i}(\R^d)- \pi'_{i}(\R^d)| + \min\{\pi_{i}(\R^d), \pi'_{i}(\R^d)\} \hat d_1(\pi_{i}, \pi'_{i}) \Big\}.\]
	The assertion now follows if we plug in the concrete forms of the $\pi_i's$, use the fact that $\pi'(\R^d) \leq  \EEE{\Xi(\R^d)} $ and apply the triangle inequality
	$$
	\mathbf{d_2}(\mathcal L \Xi, \mathsf{CP}(\pi_1',\pi_2',\dots)) \le \mathbf{d_2}(\mathcal L \Xi, \mathsf{CP}(\pi_1,\pi_2,\dots)) + \mathbf{d_2}( \mathsf{CP}(\pi_1,\pi_2,\dots), \mathsf{CP}(\pi_1',\pi_2',\dots)).
	$$
\end{proof}

\section{Process approximation for stabilizing statistics}
\label{sec:stabilizing}
In this section, we derive Poisson and compound Poisson approximations for stabilizing statistics with explicit rates of convergence. Let $\xi$ be a stationary point process with intensity $\gamma>0$ and $q$th correlation function $\rho_q$ for $q \in \N$. Recall that the latter is given by 
the identity:
	\[\E\Big[ \prod_{i\leq q} \xi(A_i)  \Big] = \int_{A_1\times\cdots \times A_q}\rho_k(x_1,\ldots,x_q)\mathrm{d}x_1\ldots \mathrm{d}x_q\] for any pairwise disjoint Borel sets $A_1,\ldots, A_q$. 
For $k \in \N$ let $g_n:(\R^d)^k \times \mathbf{N} \rightarrow \{0,1\}$ and $z:(\R^d)^k\rightarrow \R^d$ be measurable functions such that $g_n$ is translation-invariant and $z$ is translation-covariant. For a given compact set $W \subset \mathbb R^d$, we write $W_n:=n^{1/d}W$ and define
\begin{equation}
	\Xi_n[\omega]= \sum_{\xx \in \omega_{\neq}^k} \one\{z(\xx) \in W_n\}  g_n(\xx,\omega) \delta_{n^{-1/d}z(\xx)},\quad \omega \in \mathbf N, \label{def:xi}
\end{equation}
When $\omega=\xi$, we simply write $\Xi_n := \Xi_n[\xi]$ and refer to the latter as the \textit{point process of exceedances} in $W$. Due to the stationarity of $\xi$ and translation-invariance of $g_n$, we find that the intensity measure $\E[\Xi_n]$ of $\Xi_n$ is given by
$$
\E[\Xi_n](\d z)=n  \Big(\int_{(\R^d)^k} \one\{z(\xx) \in W\} \E[g_n(\xx,\xi^{\xx})]\rho_k(\xx) \d \xx\Big) \one\{z \in W\}\d z=:n \g_n \one\{z \in W\} \d z.
$$
We assume that
	\begin{align}
	\g_n>0,\quad n \in \N.\tag{{\bf M}} \label{eq:m}\nonumber
	\end{align}
Second, we require that $g_n$ is {\em exponentially stabilizing} in the spirit of \cite{BSY22}. To make this precise,  we assume that there is a \emph{stabilization radius}  $R_n(z(\xx),\om) \in [r(\xx),+\infty]$, where $r(\mathbf{x}) = \max_{i=1,\ldots, k}||x_i-z(\mathbf{x})||$; that is 
$$ g_n(\xx, \om) = g_n(\xx, \om \cap B_{R_n(z(\xx), \om)}(z(\xx))) 
$$  holds for any $\om \in \mathbf{N}$ and $\xx = (x_1,\ldots, x_k)\in \om_{\neq}^k$, where $B_r(z)$ denotes the (closed) ball centered at $z$ with radius $r$ for any $z\in \R^d$ and $r>0$. The event $\{R_n(o,\om)\le r\}$ is supposed to be measurable with respect to the $\sigma$-field $\mathcal{N}_{B_r(o)}$. We also assume for any $\xx\in \R^d$ that $\mathcal{S}_n(\xx,\cdot):\mathbf{N}\to \mathcal{F},\, \om \mapsto B_{R_n(\xx,\om)}(z(\xx))$ is a stopping set, i.e.\ for any compact sets $S \subseteq \R^d$, 
\begin{align*}
	\{\omega \in \mathbf{N} :\,\mathcal{S}_n(\xx,\omega)\subseteq S \}= 	\{\omega \in \mathbf{N} :\,\mathcal{S}_n(\xx,\omega \cap S)\subseteq S \},
\end{align*}
where $ \mathcal{F}$ denotes the set of closed subsets in $\R^d$.

\subsection{Poisson approximation}
In this section, we impose that the stabilization radius satisfies the following property: for any sequence $(r_n)$ such that $r_n\geq (\log n)^{1/2+\varepsilon}$ for some $\varepsilon>0$ and for $n$ large enough, 
\begin{equation}
	-	\limsup_{n\to \infty}\,r_n^{-1} \log \Big(\int_{(\R^d)^k}\one\{z(\xx)\in [0,1]^d\}\P(R_n(z(\xx),\xi^{\xx}) > r_n) \rho_k(\xx) \d \xx \Big) =:c_{\ms{es}}>0.\tag{{\bf S}} \label{eq:s}\nonumber
\end{equation}

\begin{thm}\label{Th:Poi} 	Let $W$ be a compact subset in $\R^d$. Let $g_n:(\R^d)^k\times \mathbf N \to \{0,1\}$ be measurable and translation-invariant and let $\xi$ be a stationary point process on $\R^d$ that satisfies the $\mathsf{EDD}$ property, has intensity $\gamma>0$ and $q$th correlation function $\rho_q$ for $1 \le q \le 2k$. We assume that \eqref{eq:m} and \eqref{eq:s} hold, define $\Xi_n := \Xi_n[\xi]$ by $\eqref{def:xi}$ and let $\pi$ be a finite measure on $W$. Then
		\begin{multline*}
			\mathbf{d_2}(\mathcal L \Xi_n,\mathsf{Po}(\pi)) \le  \hat d_1(n \g_n \mathsf{Leb}_W,\pi)+  e^{n\g_n |W| } \Big(c_1 \g_n(\log n)^d +\frac{c_2}{n}\\
			+\sum_{\ell=1}^{k} \gamma^{k+\ell}  \int \one\{z(\xx) \in W_n,\,z(\xx_{k-\ell},\yy) \in B_{b_n}(z(\xx))\}  \E[g_n(\xx,\xi^{\xx,\yy}) g_n((\xx_{k-\ell},\yy),\xi^{\xx,\yy})]  \,\rho_{k+\ell}(\xx,\yy) \,\d (\xx, \yy) \Big),
		\end{multline*}
	where the integral is over $(\R^d)^k\times (\R^d)^{\ell}$, $\xx_{k-\ell}:=(x_1,\dots,x_{k-\ell})$,  $\mathsf{Leb}_W$ denotes the  Lebesgue measure on $W$,  $b_n:=[\theta_2^{-1}(4+2\theta_0/d) +4c_{\ms{es}}^{-1}]\log n $ and $c_1, c_2$ are positive constants that do not depend on $\xi$ and $g_n$.
\end{thm}

\begin{proof}
For $s_n:=2c_{\ms{es}}^{-1}\log n$, we define $\tilde g_n(\mathbf{x},\omega):=g_n(\mathbf{x},\omega) \one\{R_n(z(\mathbf{x}),\omega)\le s_n\}$ as well as the truncated process \[\Xi_{n,\mathsf{tr}}[\omega]:=\sum_{\xx \in \omega_{\neq}^k} \one\{z(\xx) \in W_n\} \tilde g_n(\xx,\omega)  \delta_{n^{-1/d}z(\xx)}\] and set $\Xi_{n,\mathsf{tr}}:=\Xi_{n,\mathsf{tr}}[\xi]$. By the triangle inequality for the $\mathbf{d_2}$ distance, we have
	\begin{align*}
		\mathbf{d_2}(\mathcal L \Xi_n, \mathsf{Po}(\pi)) & \le \mathbf{d_2}(\mathcal L \Xi_n, \mathcal L \Xi_{n,\sf{tr}}) + 	\mathbf{d_2}(\mathcal L \Xi_{n,\mathsf tr}, \mathsf{Po}(\E[\Xi_{n,\mathsf{tr}}])) +  	\mathbf{d_2}(\mathsf{Po}(\E[\Xi_{n,\mathsf{tr}}]), \mathsf{Po}(\pi)). 
	\end{align*}
	Since  $\Xi_{n,\sf{tr}}(A) \le \Xi_n(A)$ for all Borel sets $A \subset \R^d$, we obtain from \eqref{eq:KRduality} and assumption \eqref{eq:s} that
\begin{align*}
\mathbf{d_2}(\mathcal L \Xi_n, \mathcal L \Xi_{n,\sf{tr}}) & \leq 	\E[\hat d_1(\Xi_n,\Xi_{n,\sf{tr}})] \\
& \le \E[\Xi_n(\R^d)] - \E[\Xi_{n,\sf{tr}}(\R^d)]\\
	&=\int_{(\R^d)^k} \one\{z(\xx) \in W_n\} \E[g_n(\xx,\xi^{\xx}) \one\{R_n(z(\xx),\xi^{\xx})>s_n\}] \rho_k(\xx) \d \xx\\
	& \leq  |W_n| \int_{(\R^d)^k} \one\{z(\xx) \in [0,1]^d\} \P(R_n(z(\xx),\xi^{\xx})>s_n) \rho_k(\xx) \d \xx\\
	&  \le \frac{\alpha_1}{n},
\end{align*}
for some $\alpha_1>0$, where we have used the concrete form of $s_n$. Moreover,
\begin{align*}
\mathbf{d_2}(\mathsf{Po}(\E[\Xi_{n,\mathsf{tr}}]), \mathsf{Po}(\pi))&\le \mathbf{d_2}(\mathsf{Po}(\E[\Xi_{n,\mathsf{tr}}]), \mathsf{Po}(\E[\Xi_{n}])) +\mathbf{d_2}(\mathsf{Po}(\E[\Xi_{n}]), \mathsf{Po}(\pi))\\
&\le \E[\Xi_n(\R^d)] - \E[\Xi_{n,\sf{tr}}(\R^d)] + \hat d_1(\E[\Xi_n],\pi)\\
& \leq \frac{\alpha_1}{n} + \hat d_1(n \g_n \mathsf{Leb}_W,\pi). 
\end{align*}
In what follows, we provide an upper bound for $\mathbf{d_2}(\mathcal L \Xi_{n,\mathsf{tr}}, \mathsf{Po}(\EEE{\Xi_{n,\mathsf{tr}}}))$. To do it, we apply Proposition \ref{pr:cp_app} with $A_z:=\{z\}^c$.   This gives
	\begin{multline}
		\mathbf{d_2}(\mathcal L \Xi_{n,\mathsf{tr}}, \mathsf{Po}(\EEE{\Xi_{n,\mathsf{tr}}}))\\
		\begin{split}
		& = \mathbf{d_2}(\mathcal L \Xi_{n,\mathsf{tr}}, \mathsf{CP}(\E[\Xi_{n,\mathsf{tr}}],0,0,\dots))\\
		& \le 2e^{\E[\Xi_{n,\mathsf{tr}}(\R^d)]} \Big(\int_W \E[\Xi_{n,\mathsf{tr}}(\{z\}) ] \, \E[\Xi_{n,\mathsf{tr}}](\d z)
		+ \int_W \dtv((\Xi_{n,\mathsf{tr}}^{z})_{\{z\}^c},(\Xi_{n,\mathsf{tr}})_{\{z\}^c}) \, \E[\Xi_{n,\mathsf{tr}}](\d z)\Big).
		\end{split}
		\label{eq:poid2boucom}
	\end{multline}
	Due to the stationarity of $\xi$ and translation-invariance of $g_n$, we have $\E[\Xi_{n,\mathsf{tr}}] (\{z\})=0$ for every $z \in \R^d$, implying that the first integral on the right-hand side vanishes. We let $a_n:=n^{-1/d}b_n$ and bound the total variation distance in the second integral by
	\begin{multline*}
		\dtv((\Xi_{n,\mathsf{tr}}^{z})_{\{z\}^c},(\Xi_{n,\mathsf{tr}})_{\{z\}^c})\\
		\begin{split} &\le \E[\Xi_{n,\mathsf{tr}}(B_{a_n}(z)\setminus \{z\})] +  \E[\Xi_{n,\mathsf{tr}}^z(B_{a_n}(z)\setminus \{z\})] + \dtv((\Xi_{n,\mathsf{tr}}^{z})_{B_{a_n}(z)^c},(\Xi_{n,\mathsf{tr}})_{B_{a_n}(z)^c}) \\
		& = \E[\Xi_{n,\mathsf{tr}}(B_{a_n}(z))] +  \E[\Xi_{n,\mathsf{tr}}^{z!}(B_{a_n}(z))] + \dtv((\Xi_{n,\mathsf{tr}}^{z})_{B_{a_n}(z)^c},(\Xi_{n,\mathsf{tr}})_{B_{a_n}(z)^c}),
		\end{split}
	\end{multline*}
	where $\Xi^{z!}:=\Xi^z\setminus \{z\}$. Here, we have for some constant $\alpha_2>0$,
	$$
	\E[\Xi_{n,\mathsf{tr}}(B_{a_n}(z))]\leq n\g_n |B_{a_n} \cap W|\le \alpha_2\, \gamma_n (\log n)^d.
	$$
Moreover,
	\begin{multline*}
	\int_{W} \E[\Xi_{n,\mathsf{tr}}^{z!}(B_{a_n}(z))] \E[\Xi_{n,\mathsf{tr}}](\d z) = \sum_{\ell=1}^{k} \EE\Big[\sum_{\mathbf{x}\in (\R^d)^k}\sum_{\mathbf{y}\in (\R^d)^{\ell}}  \one\{z(\xx) \in W_n\} \one\{z(\xx_{k-\ell},\yy) \in B_{n^{1/d}a_n}(z(\xx))\} \\ \times \tilde g_n(\xx,\xi) \tilde g_n((\xx_{k-\ell},\yy),\xi) \Big],
	\end{multline*}
	which gives
	\begin{multline*}
	\int_{W} \E[\Xi_{n,\mathsf{tr}}^{z!}(B_{a_n}(z))] \E[\Xi_{n,\mathsf{tr}}](\d z)=\sum_{\ell=1}^{k} \gamma^{k+\ell} \int\limits_{(\R^d)^k\times (\R^d)^{\ell}} \one\{z(\xx) \in W_n\} \one\{z(\xx_{k-\ell},\yy) \in B_{n^{1/d}a_n}(z(\xx))\} \\ \times \E[\tilde g_n(\xx,\xi^{\xx,\yy}) \tilde g_n((\xx_{k-\ell},\yy),\xi^{\xx,\yy})]  \,\rho_{k+\ell}(\xx,\yy) \,\d (\xx, \yy).
	\end{multline*}
	For $z \in W$ let $\xi^{z,\Xi_{n,\mathsf{tr}}}$ be a Palm version of $\xi$ at $z$ with respect to $\Xi_{n,\mathsf{tr}}$. Then we have $\Xi_{n,\mathsf{tr}}^{z} \stackrel{d}{=}\Xi_{n,\mathsf{tr}}[\xi^{z,\Xi_{n,\mathsf{tr}}}]$. Since $R_n$ is a stabilization radius for $\tilde{g}_n$ and since $\tilde{g}_n(z(\mathbf{x}), \omega)=0$ when $R_n(z(\mathbf{x}),\omega)>s_n$, we know from the stopping set property of $\mathcal S_n$, applied to $S=B_{s_n}(W_n\setminus B_{n^{1/d}a_n}(z))$, that $(\Xi_{n,\mathsf{tr}}^{z})_{B_{a_n}(z)^c}$ is measurable with respect to $(\xi^{z,\Xi_{n,\mathsf{tr}}})_{B_{s_n}(W_n\setminus B_{n^{1/d}a_n}(z))}$ (and analogously for $(\Xi_{n,\mathsf{tr}})_{B_{a_n}(z)^c}$), where $B_s(A):=A+B_s$ for some Borel set $A \subset \R^d$ and $s>0$, with $B_s:=B_s(o)$. This allows us to bound the total variation distance in \eqref{eq:poid2boucom} with $B_n^-:=B_{n^{1/d}w}\setminus B_{b_n}$, where $w:=\diam(W)$, by
	\begin{align*}
		\dtv((\Xi_{n,\mathsf{tr}}^{z})_{B_{a_n}(z)^c},(\Xi_{n,\mathsf{tr}})_{B_{a_n}(z)^c})& \le \dtv((\xi^{z,\Xi_{n,\mathsf{tr}}})_{B_{s_n}(W_n\setminus B_{n^{1/d}a_n}(z))}, \xi_{B_{s_n}(W_n\setminus B_{n^{1/d}a_n}(z))})\\
		&\le \dtv((\xi^{z,\Xi_{n,\mathsf{tr}}})_{B_{s_n}(z+B_n^-)}, \xi_{B_{s_n}(z+B_n^-)})\\
		&=\dtv((\xi^{o,\Xi_{n,\mathsf{tr}}})_{B_{s_n}(B_n^-)}, \xi_{B_{s_n}(B_n^-)}),
	\end{align*}
	 where the last equality follows from the stationarity of $\xi$. Let $D_{\xx}=B_{s_n}(z(\xx))$ and $h_n(\xx,\om)=\one\{z(\xx) \in W_n\}\tilde{g}_n(\xx,\om)  $. Since $h_n(\xx,\xi)$ only depends on $\xi_{D_{\xx}}$, we have
	$$\Xi_{n,\mathsf{tr}} =\sum_{\xx \in  \xi_{\neq}^k} h_n(\xx,\xi_{D_{\xx}}) \delta_{n^{-1/d}z(\xx)}.
	$$

Then we obtain from Proposition \ref{prop:mix}  with $p=2$, $q=\infty$, $B=B_1$, $A:=B_{s_n}(B_n^-)$ and from the $\mathsf{EDD}$ property of $\xi$ that \label{p:boundTVpoisson}
	\begin{multline}
	\dtv((\xi^{o,\Xi_{n,\mathsf{tr}}})_{B_{s_n}(B_n^-)}, \xi_{B_{s_n}(B_n^-)})\\
\begin{split}
	&\le \frac{4(\E[{\Xi_{n,\mathsf{tr}}}(B)^2])^{1/2}}{\E[\Xi_{n,\mathsf{tr}}(B)]} \beta(\xi_{B_{s_n+1}},\xi_{B_{s_n+1}(B_n^-)})^{1/2}\nonumber\\
	& \le \frac{4(\E[\Xi_{n,\mathsf{tr}}(B)^2])^{1/2}}{\E[\Xi_{n,\mathsf{tr}}(B)]} \left(\theta_1 [(wn^{1/d}+s_n+1)(s_n+1) ]^{\theta_0} e^{-\theta_2 (b_n-2s_n-2)}\right)^{1/2}.
	\end{split}
	\end{multline}
	This gives, for some positive constant $\alpha_3$ and for $n$ large enough, 
	\begin{multline*}
		\int_{W} \dtv((\Xi_{n,\mathsf{tr}}^{z})_{B_{a_n}(z)^c},(\Xi_{n,\mathsf{tr}})_{B_{a_n}(z)^c}) \, \E[\Xi_{n,\mathsf{tr}}](\d z)\\
	\begin{split}	
	&\le \frac{\E[\Xi_{n,\mathsf{tr}}(W)]}{\E[\Xi_{n,\mathsf{tr}}(B)]}\frac{4e^{\theta_2}\left(\theta_1 [(wn^{1/d}+s_n+1)(s_n+1) ]^{\theta_0}\right)^{1/2}}{n^{2+\theta_0/d}} (\E[\Xi_{n,\mathsf{tr}}(B)^2])^{1/2}\\
	& \le \frac{\alpha_3}{n},
	\end{split}
	\end{multline*}
which concludes the proof of Theorem \ref{Th:Poi}.
	\end{proof}

\subsection{Compound Poisson approximation}

Next we consider compound Poisson process approximation. Given a sequence of Borel sets $C_n \subset \R^d$ with $o \in C_n,\, n \in \N$, we consider the processes 
\begin{align} \label{eqn:xiell}
	\Xi_n^{(i)}:=\sum_{z \in \Xi_n} \one\{\Xi_n((C_n)_z)=i\} \delta_z=\sum_{\xx \in \xi_{\neq}^k} \one\{z(\xx)\in W_n\} \one\{\Xi_n((C_n)_{z(\xx)})=i\} g_n(\xx,\xi) \delta_{n^{-1/d}z(\xx)},\quad i \ge 1,
\end{align}
where $(C_n)_x:=x+C_n$. In particular, we have $\Xi_n=\sum_{i \ge 1} \Xi_n^{(i)}$. 

We need a slightly stronger stabilization condition than for Poisson approximation and impose that the stabilization radius satisfies the following property: for any $\ell=0,\dots,k$ and for any sequence $(r_n)$ such that $r_n\geq (\log n)^{1/2+\varepsilon}$ for some $\varepsilon>0$ and for $n$ large enough, 
\begin{multline*}
	-	\limsup_{n\to \infty}\,r_n^{-1} \log \Big(\int_{(\R^d)^{k+\ell}}   \one\{z(\xx)\in W_n\} \one\{z(\xx_{k-\ell},\yy)\in W_n \cap (C_n)_{z(\xx)}\}\\
	\times  \EEE{g_n(\xx,\xi^{\xx,\yy}) g_n((\xx_{k-\ell},\yy),\xi^{\xx,\yy}) \one\{R_n(z(\xx),\xi^{\xx,\yy}) > r_n\}} \rho_k(\xx,\yy) \d (\xx,\yy) \Big) =:c_{\ms{es}}>0.\tag{{\bf S*}} \label{eq:s*}\nonumber
\end{multline*}
In particular, when $\ell=0$, the above equation is equivalent to $\mathbf{(S)}$. Moreover, we let $c_n:=\ms{diam}(C_n),\, n \in \N,$ and assume that
\begin{align}
\limsup_{n\to \infty} \frac{c_n}{\log n}=:c_{\ms{c}} \in [0,+\infty).\tag{{\bf C}} \label{eq:c}\nonumber
\end{align}

\begin{thm} \label{Th:CP}
	Let $g_n:(\R^d)^k\times \mathbf N \to \{0,1\}$ be measurable and translation-invariant and let $\xi$ be a stationary point process on $\R^d$ that satisfies the $\mathsf{EDD}$ property, has intensity $\gamma>0$ and $q$th correlation function $\rho_q$ for $1 \le q \le 2k$. Assume that \eqref{eq:m}, \eqref{eq:c} and \eqref{eq:s*} hold and define $\Xi_n:= \Xi_n[\xi]$ by $\eqref{def:xi}$. Let $\mathsf{CP}(\pi_1,\pi_2,\dots)$ be the distribution of a compound Poisson process with parameters $\pi_1$, $\pi_2$, \ldots  Then 
	\begin{multline*}
		\mathbf{d_2}(\mathcal L \Xi_n,\mathsf{CP}(\pi_1,\pi_2,\dots)) 
		\le   e^{n \g_n |W|} \Big(c_1n \g_n^2(\log n)^d +\frac{c_2 (\log n)^d}{n}\\
		+\sum_{i \ge 1} \Big\{ |i\pi_{i}(\R^d)- \E[\Xi_n^{(i)}(\R^d)] | + \min\{\pi_{i}(\R^d), \E[\Xi_n^{(i)}(\R^d)]/i\} \hat d_1(\pi_{i},\E[\Xi_n^{(i)}]/i) \Big\}\Big),
	\end{multline*}
	where $c_1,c_2$ are positive constants that do not depend on $\xi$ and $g$.
\end{thm}
Theorem \ref{Th:CP} is more general than Theorem \ref{Th:CP} since it can be used to derive Poisson approximations. However, from a practical point of view, we apply Theorem \ref{Th:Poi} rather than Theorem \ref{Th:CP} to establish a Poisson approximation result since the latter requires less restrictive assumptions.

\begin{proof}[Proof of Theorem \ref{Th:CP}] 	
For $i \ge 1$, we define the truncated processes
\begin{equation} \label{eqn:xielltr}
	\Xi_{n,\ms{tr}}^{(i)}:=\sum_{\xx \in \xi_{\neq}^k} \one\{z(\xx)\in W_n\} \one\{\Xi_n((C_n)_{z(\xx)})=i\} g_n(\xx,\xi) \one\{\forall w \in \Xi_n \cap (C_n)_{z(\xx)}:R_n(w,\xi)\le s_n \} \delta_{n^{-1/d}z(\xx)},
\end{equation}
and set $i\pi_{i,\ms{tr}}:=\E[\Xi_{n,\ms{tr}}^{(i)}]$, with $s_n:=2c_{\ms{es}}^{-1}\log n$. We also let  $\Xi_{n,\ms{tr}}:=\sum_{i \ge 1} \Xi_{n,\ms{tr}}^{(i)}$. 

By the triangle inequality for the $\mathbf{d_2}$ distance, we have
\begin{multline*}
	\mathbf{d_2}(\mathcal L \Xi_n, \mathsf{CP}(\pi_1,\pi_2,\dots))\\
	\,\le \mathbf{d_2}(\mathcal L \Xi_n, \mathcal L \Xi_{n,\sf{tr}}) + 	\mathbf{d_2}(\mathcal L \Xi_{n,\mathsf tr}, \mathsf{CP}(\pi_{1,\ms{tr}},\pi_{2,\ms{tr}},\dots)) + \mathbf{d_2}( \mathsf{CP}(\pi_{1,\ms{tr}},\pi_{2,\ms{tr}},\dots), \mathsf{CP}(\pi_1,\pi_2,\dots)).
\end{multline*}	
Since  $\Xi_{n,\sf{tr}} \subset \Xi_n $ a.s., we find from the definition of the $\mathbf {d_2}$ distance  that
\begin{align*}
\mathbf{d_2}(\mathcal L \Xi_n, \mathcal L \Xi_{n,\sf{tr}}) & \le \E[\Xi_n(\R^d)] - \E[\Xi_{n,\sf{tr}}(\R^d)]\\
& =\int_{\R^d} \one\{z(\xx) \in W_n\}  \E[g_n(\xx,\xi^{\xx}) \one\{\exists w \in \Xi_n \cap (C_n)_{z(\xx)}:R_n(w,\xi)> s_n \} ] \rho_k(\xx)\d \xx\\
& \leq \sum_{\ell=0}^k \int_{(\R^d)^k \times (\R^d)^\ell} \one\{z(\xx)\in W_n\}  \one\{z(\xx_{k-\ell},\yy) \in W_n \cap (C_n)_{z(\xx)}\}\\
& \hspace{2.75cm} \times \E[g_n(\xx,\xi^{\xx,\yy}) g_n((\xx_{k-\ell}, \yy),\xi^{\xx,\yy}) \one\{R_n(\xx,\xi^{\xx,\yy})>s_n\}] \rho_{k +\ell}(\xx,\yy) \d(\xx,\yy)
\end{align*}
and therefore, from condition \eqref{eq:s*}, that $\mathbf{d_2}(\mathcal L \Xi_n, \mathcal L \Xi_{n,\sf{tr}})\leq \frac{\alpha_1}{n}$. Moreover, we find from \cite[Theorem 4.6]{BM02} that
 \begin{multline*}
 	\mathbf{d_2}( \mathsf{CP}(\pi_{1,\ms{tr}},\pi_{2,\ms{tr}},\dots), \mathsf{CP}(\pi_1,\pi_2,\dots))\\
 	\begin{split}
 	 & \le e^{n \gamma_n |W|} \Big(\sum_{i \ge 1}\Big\{ i |\pi_i(\R^d) - \pi_{i,\ms{tr}}(\R^d)| + \min(\pi_i,\pi_{i,\ms{tr}}) \hat d_1(\pi_i, \pi_{i,\ms{tr}}) \Big\}\Big)\\
 	& = e^{n \g_n |W|}\sum_{i \ge 1} \Big\{ |i\pi_{i}(\R^d)- \E[\Xi_n^{(i)}(\R^d)] | + \min\{\pi_{i}(\R^d), \E[\Xi_n^{(i)}(\R^d)]/i\} \hat d_1(\pi_{i},\E[\Xi_n^{(i)}]/i) \Big\}.
 	\end{split}
 \end{multline*}

In what follows, we provide an upper bound for $\mathbf{d_2}(\mathcal L \Xi_{n,\ms{tr}}, \mathsf{CP}(\pi_{1,\ms{tr}},\pi_{2,\ms{tr}},\dots))$. To do it, we apply Proposition \ref{pr:cp_app} with $A_z:=B_{a_n}(z)^c$, where $a_n:=n^{-1/d}[\theta_2^{-1}(4+2\theta_0/d) +4c_{\ms{es}}^{-1}+2c_{\ms{c}}]\log n$.  This gives
\begin{multline}
	\mathbf{d_2}(\mathcal L \Xi_{n,\ms{tr}}, \mathsf{CP}(\pi_{1,\ms{tr}},\pi_{2,\ms{tr}},\dots))\le e^{n \g_n |W|} \Big(\int_W \big(a_n+2\E[\Xi_{n,\ms{tr}}(B_{a_n}(z))] \big)\, \E[\Xi_{n,\ms{tr}}](\d z)\\
+2\sum_{i \ge 1} \int_W \dtv((\Xi_{n,\ms{tr}}^{zi})_{B_{a_n}(z)^c},(\Xi_{n,\ms{tr}})_{B_{a_n}(z)^c}) \,\E[\Xi_{n,\ms{tr}}^{(i)}](\d z)\Big).\label{eq:d2boucom}
\end{multline}
 Here, we have 
\begin{multline*}
\int_W \E[\Xi_{n,\ms{tr}}(B_{a_n}(z))] \E[\Xi_{n,\ms{tr}}](\d z)\\
\begin{split}
& =\int_{(\R^d)^k \times (\R^d)^k} \one\{z(\xx)\in W_n,\,z(\yy)\in B_{n^{1/d}a_n}(z(\xx))\cap W_n\} \E[g_n(\xx,\xi^{\xx})] \E[g_n(\yy,\xi^{\yy})] \rho(\xx) \rho(\yy) \d(\xx,\yy)\\
& \le c n \g_n^2 (\log n)^d,
\end{split}
\end{multline*}
for some constant $c>0$. 

For $i \ge 1$ and $z \in W$ let $\xi^{zi}$ be a Palm version of $\xi$ at $z$ with respect to $\Xi_{n,\ms{tr}}^{(i)}$. Then we have $\Xi_{n,\ms{tr}}^{zi} \stackrel{d}{=}\Xi_{n,\ms{tr}}[\xi^{zi}]$. Proceeding in the same spirit as in the proof of Theorem \ref{Th:Poi}, we can prove that with $B_n^-:=B_{n^{1/d}w}\setminus B_{b_n}$,
\[\dtv((\Xi_{n,\ms{tr}}^{zi})_{B_{a_n}(z)^c},(\Xi_{n,\ms{tr}})_{B_{a_n}(z)^c})  \leq \dtv((\xi^{oi})_{B_{s_n+c_n}(B_n^-)}, \xi_{B_{s_n+c_n}(B_n^-)}).\]
Let $D_{\xx}=B_{s_n+c_n}(z(\xx))$ and \[h_n^{(i)}(\xx,\om)=\one\{z(\xx) \in W_n\}\one\{\Xi_n(C_{z(\xx)})=i\} \one\{\forall w \in \Xi_n \cap (C_n)_{z(\xx)}:R_n(w,\xi)\le s_n \} g_n(\xx,\xi).\] By the stopping set property of $\mathcal S_n$ and the fact that $R_n$ is a stabilization radius, the function $h_n^{(i)}(\xx,\xi)$ is measurable with respect to $\xi_{D_\xx}$. Therefore
$$\Xi_{n,\mathsf{tr}}^{(i)} =\sum_{\xx \in  \xi_{\neq}^k} h_n^{(i)}(\xx,\xi_{D_{\xx}}) \delta_{n^{-1/d}z(\xx)}.
$$
Similarly to p. \pageref{p:boundTVpoisson}, we obtain from the $\mathsf{EDD}$ property of $\xi$ that
\begin{multline*}
\dtv((\xi^{oi})_{B_{s_n+c_n}(B_n^-)}, \xi_{B_{s_n+c_n}(B_n^-)})\\
\le \frac{4(\E[\Xi_{n,\ms{tr}}^{(i)}(B)^2])^{1/2}}{\E[\Xi_{n,\ms{tr}}^{(i)}(B)]}\left( \theta_1 [(wn^{1/d}+s_n+c_n+1)(s_n+c_n+1) ]^{\theta_0} e^{-\theta_2 (b_n-2c_n-2s_n-2)}\right)^{1/2}.
\end{multline*}
Thanks to Condition \eqref{eq:c}, this gives
\begin{multline*}
	\sum_{i \ge 1} \int_W \dtv((\Xi_{n,\ms{tr}}^{zi})_{B_{a_n}(z)^c},(\Xi_{n,\ms{tr}})_{B_{a_n}(z)^c}) \, \E[\Xi_{n,\ms{tr}}^{(i)}](\d z)\\
\begin{split}
	&\le\sum_{i \ge 1}  \frac{\E[\Xi_{n,\ms{tr}}^{(i)}(W)]}{\E[\Xi_{n,\ms{tr}}^{(i)}(B)]}\frac{4 e^{\theta_2+c_n/2 }\left(\theta_1 [(wn^{1/d}+s_n+c_n+1)(s_n+c_n+1) ]^{\theta_0}\right)^{1/2}}{n^{2+\theta_0/d}}  (\E[\Xi_{n,\ms{tr}}^{(i)}(B)^2])^{1/2}\\
	& \le \frac{\alpha_2}{n}
	\end{split}
\end{multline*}
for some positive constant $\alpha_2$ and for $n$ large enough, where we have used that $\frac{\E[\Xi_{n,\ms{tr}}^{(i)}(W)]}{\E[\Xi_{n,\ms{tr}}^{(i)}(B)]}=\frac{|W|}{|B \cap W|}$ due to the stationarity of $\xi$ and translation-invariance of $g_n$. This concludes the proof of Theorem \ref{Th:CP}.
\end{proof}

\section{Applications}
\label{sec:applications}

\subsection{Maximum nearest-neighbor distances in the Gauss-Poisson process}
In this section, we give an application of Theorem \ref{Th:Poi} in a situation where the input process $\xi$ is not Poisson. The latter is the Gauss-Poisson process and is defined as follows (see e.g. \cite{CSKM13}, Example 5.6). Let $\Phi$ be a point process defined in the following way: $\Phi$ has an isotropic distribution and is composed of zero, one or two points with probability $p_0\neq 0$, $p_1\neq 0$ and $p_2=1-(p_0+p_1)$. If $\Phi$ contains only one point then that point is the origin 0. If $\Phi$ is composed of two points then these are separated by a unit distance and have midpoint $o$. Given a stationary Poisson point process $\eta$ of intensity 1 in $\R^2$, let $\Phi(x)$, $x\in \eta$, be a family of i.i.d. point processes with the same distribution as $\Phi$. The \textit{Gauss-Poisson} process is defined as
\[\xi = \bigcup_{x\in \eta}(x+\Phi(x)).\]
Notice that when $p_2=0$, a Gauss-Poisson process is a stationary Poisson point process with intensity $p_1$. The intensity of $\xi$ is $\gamma=p_1+2p_2$. The Gauss-Poisson process was introduced by Newman \cite{N70} and investigated by Milne and Westcott \cite{MW72}. It has a potential application in statistical mechanics (see \cite{N70}, p. 350) and could be used as a model for molecular motion (see \cite{MW72} p. 169).  It  is also a Boolean model which satisfies the $\mathsf{EDD}$ property (see \cite[Lemma 3.4]{CX22}). The fact that the Palm distribution of a Gauss-Poisson process is explicitly available makes it very tractable for the application of Theorem \ref{Th:Poi}.

In all this section, we let $W_n:=[-n^{1/2}/2,n^{1/2}/2]^2$ and 
\begin{equation}
\Xi_n:=	\Xi_n[\xi]=\sum_{x \in \xi} \one\{x\in W_n\} \one\{\xi(B_{v_n}(x)\setminus \{x\})=0\} \ \delta_{n^{-1/2}x}.	
\label{eq:defxiGP}
\end{equation}

Here, the threshold $v_n$ is chosen such that $v_n \to \infty$ as $n \to \infty$ and
\begin{equation}
\lim_{n \to \infty}\mathbb{E} [\Xi_n(W_n)]= \lim_{n \to \infty}n \mathbb{P}(\xi^{o!}(B_{v_n})=0)=\tau	\label{nnintass}
\end{equation}
 for some fixed constant $\tau >0$.

 \subsubsection{The threshold $v_n$}
We provide below an explicit value of $v_n$. According to (5.53) on p. 175 in \cite{CSKM13}, we know that
\[\mathbb{P}(\xi^{o!}(B_{v})=0) = \frac{1}{p_1+2p_2}e^{-(p_1\pi v^2 + p_2(2\pi v^2 - a(v)))}\times \left\{\begin{split}   & p_1+2p_2, \quad &0\leq v<1\\
& p_1, & v\geq 1;
\end{split} \right.\]
and 
\[
a(v)=2v^2\arccos\frac{1}{2v} - \frac{1}{2}\sqrt{4v^2-1} \text{ for } v\geq 1
\]
and equals zero otherwise. Following the same computations as in \cite{C14}, p. 2949, it can be proved that the threshold\begin{equation}\label{eq:defvnGP}
v_n =  \frac{4 p_2 + \left(8p_2^2 + 8\pi (p_1+p_2) \left(\log\left(\frac{p_1}{p_1+2p_2} \right)  + \log n - \log \tau \right)\right)^{1/2}}{4\pi (p_1+p_2)}
\end{equation} 
satisfies the property \begin{equation}  \label{eq:expvnGP}\mathbb{E} [\Xi_n(W_n)] = \tau + O((\log n)^{-1/2}).\end{equation}

\subsubsection{Poisson approximation}
The following theorem claims that the point process $\Xi_n$ is asymptotically Poisson as $n$ goes to infinity. 

\begin{thm}
\label{th:d2GP}
Let $\tau>0$ and $v_n$ be as in \eqref{eq:defvnGP}. Consider the process $\Xi_n$ defined at \eqref{eq:defxiGP}.  Then 
\[
		\mathbf{d_2}(\mathcal L \Xi_n,\mathsf{Po}(\tau \mathsf{Leb}_{[0,1]^2}))  =  O((\log n)^{-1/2}).
\]
\end{thm}

As an application of the above theorem, we deal with the largest distances to the nearest neighbor. To do it, for any $x\in \R^d$, let us denote by $\mathsf{nn}(x,\xi):=\arg\min_{y\in \xi\setminus \{x\}} ||x-y||$ the nearest neighbor of $x$ in $\xi\setminus \{x\}$. Let $M_{\xi,n}^{(k)}$ be the $k$-th largest values of $\mathsf{nn}(x,\xi)$ over all $x \in \xi \cap W_n$. 

\begin{cor}
\label{cor:maxdistanceGP}
With the same notation as in Theorem \ref{th:d2GP}, we have
	$$
\lim_{n\rightarrow\infty}\mathbb{P}\big(M_{n,\xi}^{(k)}\leq v_n\big) = \sum_{i=0}^{k-1}\frac{e^{-\tau}{\tau^i}}{i!}.
	$$
\end{cor}
A similar result, only stated for the maximum, has been established in \cite[Proposition 8]{C14}. Notice that Poisson approximation for the point process of exceedances $\Xi_n$ could be also derived from \cite[Theorem 2]{C14}. However, Theorem \ref{th:d2GP} is more precise since it provides a rate of convergence.

\begin{proof}[Proof of Theorem \ref{th:d2GP}]
This will be sketched since it is mainly based on the proof of \cite[Proposition 8]{C14}. We apply Theorem \ref{Th:Poi} to $g_n(x,\omega)= \one\{\omega \cap (B_{v_n}(x)\setminus\{x\})=\es\}$, with $n\geq 1$.  The stabilization radius which we consider is defined as $R(x,\omega):=v_n$. It is straightforward that the conditions  \eqref{eq:m} and \eqref{eq:s} hold. According to Theorem \ref{Th:Poi}, we obtain 
\begin{multline}
\mathbf{d_2}(\mathcal L \Xi_n,\mathsf{Po}(\tau \lambda_{[0,1]^2})) 
\le {\frac{c_1}{n} + |n \gamma_n -\tau|+ e^{n \g_n} \Big( c_2 \gamma_n(\log n)^2 +\frac{c_3}{n}}\\
\quad + \gamma\int_{(\R^2)^2} \one\{x \in W_n,\,y \in B_{b_n}(x)\}  \P(\xi^{x!,y!}(B_{v_n}(x) \cup B_{v_n}(y))=0)  \,\rho_{2}(x,y) \, \mathrm{d}(x,y) \Big),\label{eq:d2bougin}
 \end{multline}
 where  $\xi^{x!,y!}:=\xi^{x,y}\setminus \{x,y\}$ and $\gamma_n = \PPP{\xi^{o!}(B_{v_n}) = 0}$. The integral appearing in the above expression can be expressed as 
 \begin{multline*}
 \int_{(\R^2)^2} \one\{x \in W_n,\,y \in B_{b_n}(x)\}  \P(\xi^{x!,y!}(B_{v_n}(x) \cup B_{v_n}(y))=0)  \,\rho_{2}(x,y) \, \mathrm{d}(x,y) \\
 \begin{split}
 & = \EEE{\sum_{(x,y)\in \xi_{\neq}^2}\one\{x\in W_n\}\one\{y\in B_{b_n}(x)\}\one\{(\xi\setminus\{x,y\})(B_{v_n}(x)\cup B_{v_n}(y))=0\}   }\\
 & = n\EEE{\sum_{y\in \xi^{o!}} \one\{y\in B_{b_n}(x)\}\one\{(\xi^{o!}\setminus\{y\})(B_{v_n}\cup B_{v_n}(y))=0\}\one\{||y||>v_n\}   },
 \end{split}
 \end{multline*}
where the last line comes from the fact that $\xi$ is stationary. Now, according to Equation (5.43) in \cite{CSKM13}, we know that $\xi^o\overset{d}{=}\xi\cup \Phi^o$, where $\xi$ and $\Phi^o$ are independent and where $\Phi^o$ is a Palm version of $\Phi$. Since $\xi(\{o\})=0$ and $\Phi^o(B_{v_n}^c)=0$ for $n$ large enough, this implies
\begin{multline*}
n\EEE{\sum_{y\in \xi^{o!}} \one\{y\in B_{b_n}(x)\}\one\{(\xi^{o!}\setminus\{y\})(B_{v_n}\cup B_{v_n}(y))=0\}\one\{||y||>v_n\} } \\
\begin{split}
& = n\EEE{\sum_{y\in (\xi\cup\Phi^{o!})} \one\{y\in B_{b_n}(x)\}\one\{((\xi\cup \Phi^{o!})\setminus\{y\})(B_{v_n}\cup B_{v_n}(y))=0\}\one\{||y||>v_n\} } \\
& \leq  n\EEE{\sum_{y\in \xi} \one\{y\in B_{b_n}(x)\}\one\{(\xi\setminus\{y\})(B_{v_n}\cup B_{v_n}(y))=0\}\one\{||y||>v_n\} }\\
& = n\int_{B_{b_n}}\PPP{\xi^{y!}(B_{v_n} \cup B_{v_n}(y))=0}\one\{||y>v_n||\}\mathrm{d}y\\
& = n\int_{B_{b_n}}\PPP{( ((\xi\cup \Phi^o)+y)  \setminus\{y\})(B_{v_n}\cup B_{v_n}(y))=0}\one\{||y||>v_n\}\mathrm{d}y,
\end{split}
\end{multline*}
where the last line comes from the fact that $\xi^y\overset{d}{=}(\xi\cup \Phi^o)+y$. According to Equation (5.9) in \cite{C14}, we have on the event $\{||y||>v_n\}$
\begin{equation*}
\PPP{( ((\xi\cup \Phi^o)+y)  \setminus\{y\})(B_{v_n}\cup B_{v_n}(y))=0}\leq c\PPP{\xi^{o!}(B_{v_n})=0}n^{-\alpha}=O(n^{-1-\alpha}), 
\end{equation*}
where $\alpha=\frac{p_1}{2(p_1+p_2)}$. Integrating over $B_{b_n}$ and using Equations \eqref{eq:expvnGP} and \eqref{eq:d2bougin}, we deduce that $$\mathbf{d_2}(\mathcal L \Xi_n,\mathsf{Po}(\tau \lambda_{[0,1]^2}))  =  O((\log n)^{-1/2}).$$ This concludes the proof of Theorem \ref{th:d2GP}. 
\end{proof}

\begin{proof}[Proof of Corollary \ref{cor:maxdistanceGP}]
This is a direct consequence of Theorem \ref{th:d2GP} and of the fact that $M_{n,\xi}^{(k)}\leq v_n$ if and only if $\Xi_n(W_1)\leq k-1$. 
\end{proof}

\subsection{Smallest angles in the Poisson-Delaunay tessellation}
In this section, we give an application of Theorem \ref{Th:CP} in the context of random Delaunay tessellations. Let $\omega$ be a $\sigma$-finite counting measure in $\R^2$ in general position. The Delaunay tessellation associated with $\omega$ is the unique triangulation with vertices in $\omega$ such that the circumdisk of each simplex contains no point of $\omega$ in its interior. Delaunay tessellations are a very popular structure in computational geometry~\cite{AKLK13} and are extensively used in many areas such as surface reconstruction, mesh generation, molecular modeling, and medical image segmentation, see e.g. \cite{CG06,CDS12} and \cite{OBSC00} for a wide panorama on random tessellations.   In what follows, we denote by $\del(\omega)$ the set of triangles in the Delaunay tessellation associated with $\omega$ and, for any 3-tuple of points $\mathbf{x}=(x_1,x_2,x_3) \in (\R^2)_{\neq}^3$ in general position, we denote by $\Delta(\mathbf{x})$ the triangle spanned by $\mathbf{x}$ and by $z(\mathbf{x})$ its circumcenter. We also let $\alpha_{\min}(\Delta(\mathbf{x}))$ the minimum of angles of $\Delta(\mathbf{x})$. 

In all this section, the Delaunay tessellation is based on a stationary Poisson point process $\eta$ (or on a Palm version of $\eta$) of intensity $1$ in $\RR^2$. In particular, $\eta$ satisfies the $\mathsf{EDD}$ property. Now, let $\tau>0$ be fixed and $W_n=[-n^{1/d}/2, n^{1/d}/2]$. We consider the point process
\begin{equation}
\label{eq:defxinangles}
\Xi_n:=\Xi_n[\eta]=\sum_{\mathbf{x} \in \eta_{\neq}^3}  \one\{z(\mathbf{x}) \in W_n\}  g_n(\xx,\eta)  \delta_{n^{-1/2}z(\mathbf{x})},
\end{equation}
where, for any $\omega\in \mathbf{N}$ in general position and for any  $\mathbf{x}=(x_1,x_2,x_3) \in \omega_{\neq}^3$, we write
\[g_n(\mathbf{x}, \omega) = \one\{\alpha_{\min}(\Delta(\mathbf{x}))<v_n\} \one\{\Delta(\mathbf{x}) \in \textbf{Del}(\omega)\} .\]
The threshold $v_n$ appearing in the above expression will be chosen in such a way that $\E[\Xi_n] \conv[n]{\infty} \tau$. 

\subsubsection{The threshold $v_n$}
We give below an explicit value of $v_n$. To do it, we first express $\E[\Xi_n]$ in terms of typical Delaunay triangle. In what follows, we denote the latter by $\Delta_0$. Recall that $\Delta_0$ is a random triangle whose the distribution can be defined as follows (see e.g. Section 10.4 in \cite{SW08}): for any $B\in \mathcal{B}^2$ with $|B|\in (0,\infty)$, and for any translation invariant function $h: \mathcal{K}_2\rightarrow \R_+$, where $\mathcal{K}_2$ denotes the set of convex bodies in $\R^2$,
\[\EEE{h(\Delta_0)} = \frac{1}{2 |B|} \EEE{\sum_{\mathbf{x} \in \eta_{\neq}^3  } \one\{z(\mathbf{x}) \in B\}  h(\Delta(\mathbf{x})) \one\{\Delta(\mathbf{x}) \in \mathbf{Del}(\mathbf{x})\}  }.\]
Since $|W_n|=n$, we have
\[\EEE{\Xi_n} = 2n\PPP{\alpha_{\min}(\Delta_0)<v_n}.\]
According to \cite{Mardia77}, it is known that  $\alpha_{\min}(\Delta_0)$ has a Lebesgue density given by
$$
f(t)=\frac{4}{\pi}  \sin(t)  \left((\pi-3 t)\cos (t)+\sin(3 t)\right) \I\{0 \le t \le \frac{\pi}{3}\}.
$$
This implies $\PPP{\alpha_{\min}(\Delta_0)<v} = 2v^2 -\tfrac{1}{2}v^4+o(v^4)$ and therefore $\EEE{\Xi_n} = 4nv_n^2 - nv_n^4 + o(nv_n^4)$. In particular, to ensure that $\E[\Xi_n] \conv[n]{\infty} \tau$, it is sufficient to take 
\begin{equation}
\label{eq:defvnangles}
v_n=\frac{1}{2}\tau^{1/2}n^{-1/2}.
\end{equation}
In particular, with this choice, we have \begin{equation} \label{eq:secondordertypical} 2n\PPP{\alpha_{\min}(\Delta_0)<v_n} = \tau + O(n^{-1}).\end{equation}

\subsubsection{Compound Poisson approximation}

\begin{thm}
\label{th:d2angles}
Let $\eta$ be a Poisson point process of intensity 1 in $\R^2$. Let $\tau>0$ be fixed, $\Xi_n$ as in \eqref{eq:defxinangles} and $v_n$ as in \eqref{eq:defvnangles}. Moreover, let $\mathsf{CP}(\pi_1,\pi_2,\dots)$ be the distribution of a compound Poisson process with parameters $\pi_1$, $\pi_2$, \ldots, with $\pi_1(B) = \frac{1}{2}\tau |B\cap W_1|$, $\pi_2(B) = \frac{1}{4}\tau |B\cap W_1|$ and $\pi_i(B)=0$ for any $i\geq 3$ and $B\in \mathcal B^2$.  Then, 
\[
	\mathbf{d_2}(\mathcal L \Xi, \mathsf{CP}(\pi_1,\pi_2,\dots)) =  O((\log n)^2/n) \quad \text{as } n \to \infty.
\]
\end{thm}

The above theorem can be interpreted as follows. When the size of the window $W_n=[-n^{1/2}/2, n^{1/2}/2]^2$ goes to infinity, the point process of exceedances $\Xi_n$ converges to a compound Poisson point process $\psi$. Using the identification described in Section \ref{sec:basicscompound}, the point process $\psi$ can be seen as a family of points $\{(x_k,m_k):k\geq 1\}\subset \R^d\times \mathbb{N}^*$,  where  $\{x_k:k\geq 1\}$ is a stationary Poisson point process in $W_1$ of intensity $\gamma=\tfrac{3\tau}{4}$ and where the $m_k$'s are i.i.d., independent of the $x_k$'s., and with distribution $\mathbb{Q}(1)=\tfrac{2}{3}$, $\mathbb{Q}(2)=\tfrac{1}{3}$, $\mathbb{Q}(i)=0$ for any $i\geq 3$. In particular, 
\[\PPP{\min_{\mathbf{x}\in \eta_{\neq}^3: z(\mathbf{x})\in W_n} \alpha_{\min}(\Delta(\mathbf{x}))\leq v_n} \conv[n]{\infty} \PPP{\Phi(\R^2)=0} = e^{-\frac{3}{4}\tau}.\]
In the sense of  \cite{CR18}, we obtain that the extremal index for the small angles exists and that it is $\theta=\tfrac{3}{4}$. Intuitively, it means that the mean size of a typical cluster of exceedances is $\theta^{-1}=\tfrac{4}{3}=\mathbb{Q}(1)+2\mathbb{Q}(2)$. To the best of our knowledge,  except a recent work by Basrak \textit{et al.} \cite{BMP22} on the $k$th-nearest neighbor graph, it is the only (non-trivial) example for which the extremal index and the cluster size distribution can be made explicit in the context of Stochastic Geometry. Notice also that in \cite[Section 5]{CR18}, a potential applicability of the main results of that paper to Poisson-Delaunay tessellations is discussed. However, our  method is more general, since it establishes compound Poisson approximation in the $\mathbf{d_2}$ distance with an explicit rate of convergence. 

\begin{proof}[Proof of Theorem \ref{th:d2angles}]
We apply Theorem \ref{Th:CP} to the  function $g_n: (\R^2)^3\times \mathbf{N}\rightarrow\{0,1\} $ defined, for any $\omega\in \mathbf{N}$ (in general position) and for any $\mathbf{x}=(x_1,x_2,x_3) \in \omega_{\neq}^3$, by \[g_n(\mathbf{x},\omega) = \one\{\alpha_{\min}(\Delta(\mathbf{x}))<v_n\} \one\{\Delta(\mathbf{x}) \in \mathbf{Del}(\omega)\}.  \] Condition ({\bf C}) holds by taking $C_n=B_{\log n}$. Moreover, Condition ({\bf M}) is satisfied since \[\gamma_n = \PPP{\alpha_{\min}(0) < v_n}\] is positive. To check Condition \eqref{eq:s*}, we consider the following radius of stabilization:
\[R(z(\mathbf{x}), \omega) = \tilde{R}(\mathbf{x})\one\{\Delta(\mathbf{x})\in \mathbf{Del}(\omega)\},\]
where  $\tilde{R}(\mathbf{x})$ denotes the circumradius of $\Delta(\mathbf{x})$. Let $\mathbf{y}\in (\R^2)^\ell$, with $\ell=0,\ldots, 3$. It follows from the definition of $\mathbf{Del}(\eta^{\mathbf{x},\mathbf{y}})$ that 
\begin{align*}\PPP{\Delta(\mathbf{x})\in \mathbf{Del}(\eta^{\mathbf{x}, \mathbf{y}})} & \leq \PPP{\Delta(\mathbf{x})\in \mathbf{Del}(\eta^{\mathbf{x}})}\\
& = \PPP{\eta\cap B(\mathbf{x}) = \emptyset}\\
&  = e^{-\pi (\tilde{R}(\mathbf{x}))^2},
\end{align*} where $B(\mathbf{x})$ denotes the circumball of $\Delta(\mathbf{x})$. Therefore, for any $r>0$, we have 
\[\PPP{R(z(\mathbf{x}), \eta) > r } \leq  e^{-\pi\tilde{R}(\mathbf{x})^2 }\one\{\tilde{R}(\mathbf{x})>r\} \leq e^{-\pi r^2}. \]
In particular, for any sequence $(r_n)$, we have
\begin{equation*}
\int_{(\R^2)^{k+\ell}}   \one\{z(\xx)\in W_n\} \one\{z(\xx_{k-\ell},\yy)\in W_n \cap (C_n)_{z(\xx)}\}\PPP{R_n(z(\xx),\xi^{\xx,\yy}) > r_n} \d (\xx,\yy) \leq cn^2 e^{-\pi r_n^2}.
\end{equation*}
When $r_n\geq (\log n)^{1/2+\varepsilon}$ for some $\varepsilon>0$, the right-hand side is smaller than $e^{-\pi r_n}$ for $n$ large enough. This proves that Condition  ({\bf S*}) holds.

Now, applying Theorem \ref{Th:CP}, we get	
\begin{multline*}
		\mathbf{d_2}(\mathcal L \Xi_n,\mathsf{CP}(\pi_1,\pi_2,\dots)) \\
		{\le \frac{c_1}{n} +  e^{n \g_n } \Big(\sum_{i \ge 1} \Big\{ |i\pi_{i}(\R^2)- \E[\Xi_n^{(i)}(\R^2)] | + \min\{\pi_{i}(\R^2), \E[\Xi_n^{(i)}]/i\} \hat d_1(\pi_{i},\E[\Xi_n^{(i)}]/i) \Big\}}\\
		+c_2n \g_n^2(\log n)^2 +\frac{c_3(\log n)^2}{n}\Big),
	\end{multline*}	
	 In the above equation, the point process $\Xi_n^{(i)}$ is defined as
\[ \Xi_n^{(i)} = \sum_{\mathbf{x}\in \eta_{\neq}^3} \one\{z(\mathbf{x})\in W_n\}\one\{ \alpha_{\min}(\Delta(\mathbf{x}))<v_n\} \one\{\Xi_n(C_{\mathbf{x}})=i\} \delta_{n^{-1/2}z(\mathbf{x})},  \]
where $C_{\mathbf{x}}=B_{\log n}(z(\mathbf{x}))$. Now, because $2n \PPP{\alpha_{\min}(0) < v_n}\conv[n]{\infty}\tau $, we have
\begin{multline}
\label{eq:boundd2angles}
		\mathbf{d_2}(\mathcal L \Xi_n,\mathsf{CP}(\pi_1,\pi_2,\dots)) \\
		\le c\,\Big(\sum_{i \ge 1} \Big\{ |i\pi_{i}(\R^2)- \E[\Xi_n^{(i)}(\R^2)] | + \min\{\pi_{i}(\R^2), \E[\Xi_n^{(i)}(\R^2)]/i\} \hat d_1(\pi_{i},\E[\Xi_n^{(i)}]/i) \Big\}+\frac{(\log n)^2}{n}\Big),
	\end{multline}
To deal with the terms appearing in the sum, we first give an estimate of 	$\E[\Xi_n^{(i)}(A)]$ for any Borel subset $A\subset \R^2$. According to the multivariate Mecke equation \cite[Theorem 4.4]{LastPenrose}, we have
\begin{align*}
\E[\Xi_n^{(i)}(A)] & = \E\Big[  \sum_{\mathbf{x}\in \eta^3_{\neq}}  \one\{ \alpha_{\min}(\Delta(\mathbf{x}))<v_n\} \one\{\Xi_n(B_{\log n}(z(\mathbf{x})))=i\}\one\{ z(\mathbf{x})\in (A_n\cap W_n)\}   \Big]\\
& = \int_{(\R^2)^3}\PPP{  \alpha_{\min}(\Delta(\mathbf{x}))<v_n, \Xi_n[\eta_{\R^2\setminus B(\mathbf{x})}\cup\{\mathbf{x}\}] (B_{\log n}(z(\mathbf{x}))=i  }\one\{z(\mathbf{x})\in (A_n\cap W_n)\} \mathrm{d} \mathbf{x},
\end{align*}
where $\eta_{\R^2\setminus B(\mathbf{x})}$ is a Poisson point process of intensity $1$ in $\R^2\setminus B(\mathbf{x})$ and where $A_n=n^{1/2}A$. In the above expression, with a slight abuse of notation, we have written $\{\xx\}=\{x_1,x_2,x_3\}$ for any $\xx=(x_1,x_2,x_3)\in (\R^2)^3$. The integral can be interpreted in terms of typical Delaunay triangle. Indeed, let $R$ and $\mathbf{U}=(U_1,U_2,U_3)$ be the circumradius and the angles of the vertices of $\Delta_0$. Then $R$ and $\mathbf{U}=(U_1,U_2,U_3)$ are by \cite[Theorem 10.4.4]{SW08} independent random variables with distributions
\begin{equation} \label{eq:typicaldelaunayR} \PPP{R\leq s} = 2(n \pi)^2\int_0^sr^{3}e^{-n\pi r^2}\mathrm{d}r\end{equation} and
\begin{equation} \label{eq:typicaldelaunayU}\PPP{\mathbf{U} \in E} = \frac{1}{12\pi^2}\int_{\SSS}\cdots \int_{\SSS}|\Delta(\mathbf{u})|\ind{\mathbf{u}\in E}\sigma (\mathrm{d}\mathbf{u})\end{equation} for any $s\geq 0$ and for any measurable set $E\subset \SSS^{3}$, where $\SSS$ denotes the unit circle and where $\sigma$ denotes the uniform distribution on $\SSS$ with normalization $\sigma(\SSS)=2\pi$. Integrating over $z\in (A_n\cap W_n)$ and using the fact that $\Delta_0 \overset{d}{=}\Delta(R\mathbf{U})$, we obtain
\begin{align}
\label{eq:estimateexiangles}
\E[\Xi_n^{(i)}(A)] & = 2n |A\cap W_1|\PPP{\alpha_{\min}(\Delta(R\mathbf{U}))<v_n, \Xi_n[\eta_{\R^2\setminus B_R} \cup \{R\mathbf{U}\}](B_{\log n})=i}\notag\\
& = 2n|A\cap W_1|\PPP{\alpha_{\min}(\Delta_0)<v_n}p_n(i)\notag\\
& = (\tau + \varepsilon_n) |A\cap W_1| \ p_n(i),
\end{align}
where
\begin{equation}
	p_n(i) = \PPP{ \left.    \Xi_n[\eta_{\RR^2\setminus B_R} \cup \{R\mathbf{U}\}] (B_{\log n})= i \right|  \alpha_{\min}(\Delta(R\mathbf{U}))<v_n   }\label{eqn:pin}
\end{equation}
and where $(\varepsilon_n)$ is some sequence, only depending on $n$, such that $\varepsilon_n\conv[n]{\infty}0$. In Equation \eqref{eqn:pin}, conditional on $R$, the Poisson point process $\eta_{\RR^2\setminus B_R}$ is independent of $\mathbf{U}$. The following proposition gives the order of $p_n(i)$ for each $i\geq 3$ as $n$ goes to infinity.

\begin{prop}
 \label{th:angles}
With the above notation,
\begin{enumerate}
\item $p_n(1) = 1/2 + O((\log n)^2/n)$;
\item  $p_n(2) = 1/2 + O((\log n)^2/n)$;
\item $\sum_{i\geq 3}p_n(i) = O((\log n)^2/n)$.
\end{enumerate} 
\end{prop}
Theorem \ref{th:d2angles} is a direct consequence of the above proposition, Equations \eqref{eq:boundd2angles}, \eqref{eq:estimateexiangles} and of the fact that $\pi_i(\R^2)=0$ for any $i\geq 3$.  It remains to prove Proposition \ref{th:angles}.
\end{proof}

\begin{proof}[Proof of Proposition \ref{th:angles}]
\textit{Proof of (3)}.   We start with this assertion since it is the simplest case.  Roughly, the latter can be explained by the following heuristic argument. If a Delaunay triangle $\Delta=\Delta(U_1,U_2,U_3)$ has a small angle, say  $\alpha(\overrightarrow{U_3U_1}, \overrightarrow{U_3U_2})$, of order $v_n \eq[n]{\infty} cn^{-1/2}$ then  the opposite side, namely $[U_1,U_2]$, has a length of order $c n^{-1/2}\times n^{-1/2}=c n^{-1}$, which corresponds to the minimal interpoint distance. The other triangle which shares the same edge has also an angle of order $n^{-1/2}$; and thus it is possible that clusters are of size at least 2. Clusters cannot be of size $i\geq 3$, since this would imply that there exists another triangle, whose edges are not $[U_1,U_2]$, with a small angle. This should imply that there is in the neighborhood another point which is distant to $cn^{-1}$ to its nearest neighbor, which is not possible with high probability (see e.g. \cite{H82}). To prove (3), we first notice that 
\begin{align*}
\sum_{i \ge 3}	p_n(i) & \eq[n]{\infty}  \frac{2}{\tau} n \PPP{  \Xi_n[\eta_{\RR^2\setminus B_R} \cup \{R\mathbf{U}\}] (B_{\log n}) \ge 3,  \alpha_{\min}(\Delta(R\mathbf{U}))<v_n } \\
	& \leq  \frac{6}{\tau} n \PPP{ \Xi_n[\eta_{\RR^2\setminus B_R} \cup \{R\mathbf{U}\}] (B_{\log n})\ge  3,  \alpha_3<v_n},
\end{align*}
where  $\alpha_3=\alpha_3(\Delta(R\mathbf{U}))$ denotes the measure of the angle between $\overrightarrow{U_3U_1}$ and $\overrightarrow{U_3U_2}$. To prove that the above term converges to 0, the main task is to observe the following: if there exist three exceedances, then necessarily there exists a triangle $\Delta(x_1,x_2,x_3)$ with $\alpha_3<v_n$ and another one with edges which are different  from $[x_2,x_3]$. More precisely, we have
\begin{equation}
	\label{eq:boundpi3}
	 n \PPP{ \Xi_n[\eta_{\RR^2\setminus B_R} \cup \{R\mathbf{U}\}] (B_{\log n}) \geq 3,  \alpha_3<v_n} \leq p_0^{(n)}+p_{1;1}^{(n)}+p_{1;2}^{(n)}+p_{1;3}^{(n)}+p_{2;1-3}^{(n)}+p_{2;2-3}^{(n)},
\end{equation}
where
\begin{multline*}
	p_0^{(n)} = n\PP\Big(\exists (x_4,x_5,x_6)\in (\eta_{\RR^2\setminus B_R})^3: \Delta(x_{4:6})\in \del(\eta_{\RR^2\setminus B_R}\cup\{R\mathbf{U}\}),\\ 
	z(x_4,x_5,x_6)\in B_{\log n},  \alpha_{\min}(\Delta(x_{4:6}))<v_n, \alpha_3<v_n\Big),
\end{multline*}
\begin{multline*}
	p_{1;1}^{(n)} = n\PP\Big(\exists (x_4,x_5)\in (\eta_{\RR^2\setminus B_R})^2: \Delta(x_4,x_5,RU_1)\in \del(\eta_{\RR^2\setminus B_R}\cup\{R\mathbf{U}\}),\\ 
	z(x_4,x_5,RU_1)\in B_{\log n},  \alpha_{\min}(\Delta(x_4,x_5,RU_1))<v_n, \alpha_3<v_n\Big),
\end{multline*}
\begin{multline*}
	p_{1;2}^{(n)} = n\PP\Big(\exists (x_4,x_5)\in (\eta_{\RR^2\setminus B_R})^2: \Delta(x_4,x_5,RU_2)\in \del(\eta_{\RR^2\setminus B_R}\cup\{R\mathbf{U}\}),\\ 
	z(x_4,x_5,RU_2)\in B_{\log n},  \alpha_{\min}(\Delta(x_4,x_5,RU_2))<v_n, \alpha_3<v_n\Big),
\end{multline*}
\begin{multline*}
	p_{1;3}^{(n)} = n\PP\Big(\exists (x_4,x_5)\in (\eta_{\RR^2\setminus B_R})^2: \Delta(x_4,x_5,RU_3)\in \del(\eta_{\RR^2\setminus B_R}\cup\{R\mathbf{U}\}),\\ 
	z(x_4,x_5,RU_3)\in B_{\log n},  \alpha_{\min}(\Delta(x_4,x_5,RU_3))<v_n, \alpha_3<v_n\Big),
\end{multline*}
\begin{multline*}
	p_{2;1-3}^{(n)} = n\PP\Big(\exists x_4\in \eta_{\RR^2\setminus B_R}: \Delta(x_4,RU_1,RU_3)\in \del(\eta_{\RR^2\setminus B_R}\cup\{R\mathbf{U}\}),\\ 
	z(x_4,RU_1,RU_3)\in B_{\log n},  \alpha_{\min}(\Delta(x_4,RU_1,RU_3))<v_n, \alpha_3<v_n\Big),
\end{multline*}
\begin{multline*}
	p_{2;2-3}^{(n)} = n\PP\Big(\exists x_4\in \eta_{\RR^2\setminus B_R}: \Delta(x_4,RU_2,RU_3)\in \del(\eta_{\RR^2\setminus B_R}\cup\{R\mathbf{U}\}),\\ 
	z(x_4,RU_2,RU_3)\in B_{\log n},  \alpha_{\min}(\Delta(x_4,RU_2,RU_3))<v_n, \alpha_3<v_n\Big).
\end{multline*}
In the above equations and in the following, the terms $x_i$'s are assumed to be different. We only deal with the term $p_0^{(n)} $ since the other one can be dealt with in a similar way and because it is the largest term (due to the fact that it concerns triangles with no common point). To do it, we write
\begin{equation*}
	p_0^{(n)}  =  n\EE\Big[\sum_{(x_4,x_5,x_6)\in (\eta_{\RR^2\setminus B_R})^3} \ind{\Delta(x_4,x_5,x_6)\in \del(\eta_{\RR^2\setminus B_R}\cup\{R\mathbf{U}\})}\ind{z(x_4,x_5,x_6)\in B_{\log n}} \ind{\alpha_{\min}(\Delta(x_4,x_5,x_6))<v_n}\ind{\alpha_3<v_n}\Big].
\end{equation*}
From the multivariate Mecke equation and  \eqref{eq:typicaldelaunayR}, \eqref{eq:typicaldelaunayU}, we have 
\begin{multline*}
	p_0^{(n)} = c n\int_{(\RR^2)^3}\int_{\RR_+}\int_{\SSS^3}e^{-|B_r \cup B(x_1,x_2,x_3)| } r^3 |\Delta(\mathbf{u})|\one\{z(x_4,x_5,x_6)\in B_{\log n}\} \\
	\times  \one\{\alpha_{\min}(\Delta(x_4, x_5, x_6))<v_n\}\one\{\alpha_3<v_n\}\sigma(\mathrm{d}\mathbf{u})\mathrm{d}r\mathrm{d}x_4\mathrm{d}x_5\mathrm{d}x_6.
\end{multline*}
 Applying the Blaschke-Petkantschin type of formula  \cite[Theorem 7.3.1]{SW08}, we get
\[
	p_0^{(n)} = c n \int_{B_{\log n}}\int_{\RR_+^2}\int_{\SSS^3\times \SSS^3} e^{-|B_r\cup B_{r'}(z')|}r^3 (r')^3 |\Delta(\mathbf{u})||\Delta(\mathbf{u}')|\ind{\alpha_3<v_n} \ind{\alpha_6<v_n} \sigma(\mathrm{d}\mathbf{u})\sigma(\mathrm{d}\mathbf{u}')\mathrm{d}r\mathrm{d}r'\mathrm{d}z',
\]
where $\mathbf{u}'=(u_4,u_5,u_6)\in \SSS^3$ and where $\alpha_6$ denotes the value of the angle between $\overrightarrow{u'_6u'_4}$ and $\overrightarrow{u'_6u'_5}$. Assuming without loss of generality that $r\leq r'$ and using the fact that $|B_r\cup B_{r'}(z')|\geq \pi r^2$, we have
\[
	p_0^{(n)} \leq c n \left(\int_{\SSS^3} |\Delta(\mathbf{u})|\ind{\alpha_3<v_n}  \sigma(\mathrm{d}\mathbf{u})\right)^2 \int_{B_{\log n}}\int_{\RR_+^2} e^{-\pi r^2}r^3 (r')^3 \ind{r'\leq r} \mathrm{d}r\mathrm{d}r'\mathrm{d}z' .
\]
Besides, according to \eqref{eq:typicaldelaunayU}, we know that 
\[ \int_{\SSS^3} |\Delta(\mathbf{u})|\ind{\alpha_3<v_n}  \sigma(\mathrm{d}\mathbf{u}) = c\PPP{\alpha_{\min}(\Delta(R\mathbf{U}))<v_n} = c\PPP{\alpha_{\min}(\Delta_0)<v_n} \eq[n]{\infty}c n^{-1}.\] This implies $p_0^{(n)}  = O\left((\log n)^2/n\right)$. In the same way, we can prove that the other terms appearing in \eqref{eq:boundpi3} are $O\left(1/n\right)$. This concludes the proof of (3). 

\textit{Proof of (1) and (2)}. It is sufficient to deal with (1) since $\sum_{i=1}^\infty p_n(i)=1$. To do it, we first observe that
\[
p_n(2) = \frac{\EEE{\one\{   \Xi_n[\eta_{\RR^2\setminus B_R} \cup \{R\mathbf{U}\}] (B_{\log n})= 2   \} \one\{\alpha_{\min}(\Delta(R\mathbf{U}))<v_n\}}}{\PPP{\alpha_{\min}(\Delta(R\mathbf{U}))<v_n}}.
\]
Moreover,  the Palm version of $\Xi_n=\Xi_n[\eta]$ satisfies the following equation:
\[\Xi_n^0 \overset{d}{=} \Xi_n[\eta_{\RR^2\setminus B_R}\cup \{R\mathbf{U}\}]\cup\{0\}.\]
This together with  \eqref{eq:defpalm} gives that
\[p_n(2) = \frac{\EEE{\sum_{z \in \Xi_n \cap W_n} \I\{\Xi_n(z+B_{\log n})=2\}}}{2n\PPP{\alpha_{\min}(\Delta(R\mathbf{U}))<v_n}}.\]
Since we will prove below that the numerator converges to some positive constant, it follows from \eqref{eq:secondordertypical} that
\[p_n(2) = \frac{1}{\tau} \EEE{\sum_{z \in \Xi_n \cap W_n} \I\{\Xi_n(z+B_{\log n})=2\}} + O(1/n).\]
Now we count all points in $z\in \Xi_n \cap W_n$, for which there is another triangle with minimum angle smaller than $v_n$ and with circumcenter in $z+B_{\log n}$. Using the same arguments as in the proof of (3), we can easily show that it is unlikely that these two triangles have different minimum edges and, therefore, different centres. Thus 
\begin{figure}
\includegraphics[scale=0.35]{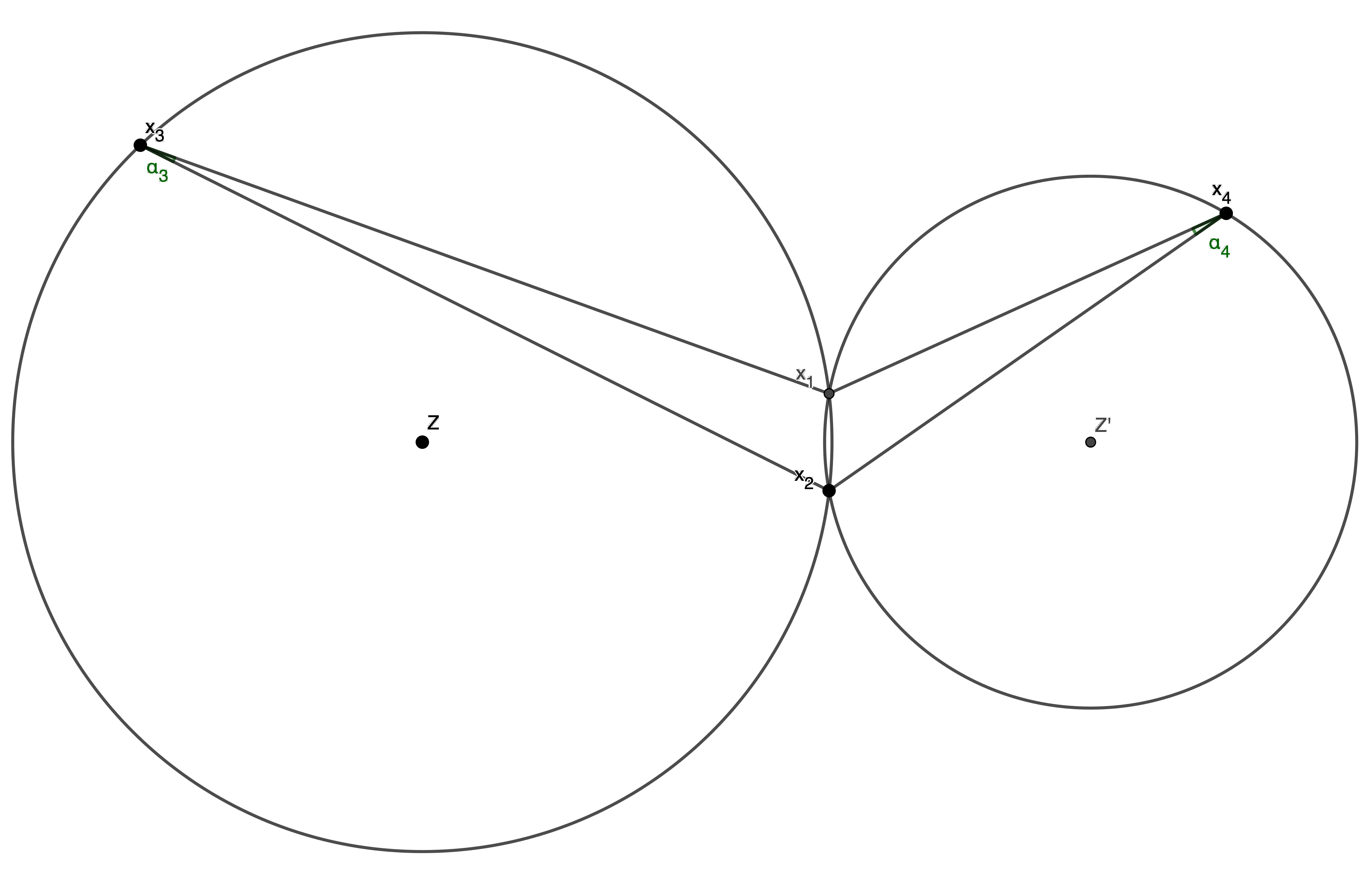}
\caption{\label{fig:angles} Two Delaunay triangles with small angles}
\end{figure}
\begin{multline*}
p_n(2) = \frac{1}{2\tau} \E\Big[\sum_{\mathbf{x}\in \eta_{\neq}^3} \sum_{x_4 \in \eta \setminus \{\mathbf{x}\}} \ind{\eta \cap B(\mathbf{x})=\emptyset}  \ind{\eta \cap B(x_1,x_2,x_4)=\emptyset}  \ind{z(\mathbf{x})\in W_n} \one\{z(x_1,x_2,x_4) \in (z(\mathbf{x}) + B_{\log n})\cap W_n\}\\
	\times \ind{\alpha_3<v_n}   \ind{\alpha_3=\alpha_{\min}(\Delta(\mathbf{x}))} \ind{\alpha_4<v_n}\ind{\alpha_4=\alpha_{\min}(\Delta(x_{1},x_2,x_4))} \Big] + O\left((\log n)^2/n \right),
\end{multline*}
where $\alpha_3$ is the value of the angle with apex $x_3$ in the triangle with vertices $x_1,x_2,x_3$ and $\alpha_4$ is the value of the angle with apex $x_4$ in the triangle with vertices $x_1,x_2,x_4$ (see Figure \ref{fig:angles}).  To do it, we apply the multivariate Mecke equation. This gives 
\begin{multline*}
p_n(2) =  \frac{1}{2\tau} \int_{(\mathbb R^2)^4} e^{-|B(\mathbf{x})\cup B(x_1,x_2,x_4)|} \one\{x_3,x_4 \notin B(\mathbf{x}) \cup B(x_1,x_2,x_4)\}  \one\{z(\mathbf{x})\in W_n\} \\\times \one\{\alpha_3<v_n\}\one\{\alpha_4<v_n\}\one\{\alpha_3=\alpha_{\min}(\Delta(\mathbf{x}))\} \one\{\alpha_4=\alpha_{\min}(\Delta(x_{1},x_2,x_4))\}\mathrm{d}\mathbf{x}\mathrm{d}x_4 + O((\log n)^2/n). 
\end{multline*}
In the above equation, the event $\{z(x_1,x_2,x_4) \in (z(\mathbf{x}) + B_{\log n})\cap W_n\}$ has disappeared because, on the complement of this event, it is unlikely that $\eta\cap B(x_1,x_2,x_4)=\emptyset$. Now, to deal with the right-hand side, we apply the transformation defined by the differentiable mapping 
\begin{align}
T:\R^2 \times \R_+ \times \SSS \longrightarrow (\R^2)^2,\label{trafo}
\end{align}
which is defined by
$$
(z,r_0,u) \longmapsto (z+r_0u,z-r_0u)=:(x_1,x_2). 
$$
 The function $T$ is bijective up to a set of measure zero. It follows by a standard argument that introduces local coordinates (see the proof of Theorem 7.3.1 in \cite{SW08}) that $T$ has Jacobian $4r_0$. Writing $a(z\pm r_0 u, x_3,x_4)$ for the area of $B(x_1,x_2,x_3)\cup B(x_1,x_2,x_4)$ with $x_1=z+ru$ and $x_2=z-ru$, we have 
\begin{multline*}
p_n(2) =  \frac{2}{\tau} \int_{(\mathbb R^2)^2}  \int_0^\infty \int_{\SSS} \int_{W_n} e^{-a(z\pm r_0 u, x_3,x_4)} \one\{x_3,x_4 \notin B(z\pm r_0u,x_3) \cup B(z\pm r_0u,x_4)\} \one\{\alpha_3<v_n\}  \\
 \times \one\{\alpha_4<v_n\}\one\{\alpha_3=\alpha_{\min}(\Delta(z\pm r_0u,x_3))\} \one\{\alpha_4=\alpha_{\min}(\Delta(z\pm r_0u,x_4))\}r_0\,\mathrm{d}z\,\sigma(\mathrm{d}u)\,\mathrm{d}r_0\,\mathrm{d}x_3\mathrm{d}x_4  + O(\log n)^2/n.
\end{multline*}
Integrating  $(z,u)$ over $W_n\times \SSS$ and denoting by $e_1=(1,0)$ the first vector of the standard basis of $\R^2$, we get
\begin{multline*}
p_n(2) =  \frac{4\pi n}{\tau} \int_{(\mathbb R^2)^2}  \int_0^\infty  e^{-a(\pm r_0 e_1, x_3,x_4)} \one\{x_3,x_4 \notin B(\pm r_0e_1,x_3) \cup B(\pm r_0e_1,x_4)\} \one\{\alpha_3<v_n\}  \\
 \times \one\{\alpha_4<v_n\}\one\{\alpha_3=\alpha_{\min}(\Delta(\pm r_0e_1,x_3))\} \one\{\alpha_4=\alpha_{\min}(\Delta(\pm r_0e_1,x_4))\}r_0\,\mathrm{d}r_0\,\mathrm{d}x_3\mathrm{d}x_4+ O(\log n)^2/n. 
\end{multline*}
	Note that the conditions $x_3,x_4 \notin B(\pm r_0e_1,x_3) \cup B(\pm r_0e_1,x_4)$ is satisfied if and only if $x_3 \in \mathbb R \times \mathbb R_+$ and $x_4 \in \mathbb R \times \mathbb R_-$ (or vice versa). Therefore
\begin{align*}
	p_n(2) = \frac{8\pi n}{\tau} \int_{(\mathbb R \times \mathbb R_+) \times (\mathbb R \times \mathbb R_-)}&  \int_0^\infty   e^{-a(\pm r_0 e_1, x_3,x_4)} \one\{x_3,x_4 \notin B(\pm r_0e_1,x_3) \cup B(\pm r_0e_1,x_4)\} \\
	&\quad	\times\one\{\alpha_3<v_n\}   \one\{\alpha_4<v_n\}\one\{\alpha_3=\alpha_{\min}(\Delta(\pm r_0e_1,x_3))\}\\ &\quad \times \one\{\alpha_4=\alpha_{\min}(\Delta(\pm r_0e_1,x_4))\}r_0\,\mathrm{d}r_0\,\mathrm{d}x_3\mathrm{d}x_4+ O(\log n)^2/n.
\end{align*}
Here we apply twice the Blaschke-Petkantschin formula from Theorem 5 in \cite{Nikitenko} and obtain
\begin{align*}
	\frac{32\pi n}{\tau} \int_{\frac{2r_0}{\sqrt 3}}^\infty \int_{\frac{2r_0}{\sqrt 3}}^\infty &\int_0^\infty   \one\{\arcsin(r_0/r)\le v_n\}   \Big((\pi-3\arcsin(r_0/r))+\frac{r\sin(3 \arcsin(r_0/r))}{\sqrt{r^2-r_0^2}}\Big)\\
	&\times \one\{\arcsin(r_0/s)\le v_n\}  \Big((\pi-3\arcsin(r_0/s))+\frac{s\sin(3\arcsin(r_0/s))}{\sqrt{s^2-r_0^2}}\Big)\\
	&\times e^{-na( r_0,r,s)}r s r_0 \,\,\mathrm{d}s\,\mathrm{d}r\mathrm{d}r_0+ O(\log n)^2/n,
\end{align*}
where $a(r_0,r,s):=|B_r(\sqrt{r^2-r_0^2}e_2) \cup B_s(-\sqrt{s^2-r_0^2}e_2)|$, with $e_2=(0,1)$. Now, consider the change of variables $r=\frac{r_0}x$ and $s=\frac{r_0}{y}$. Using the fact that $a(r_0,xr_0,yr_0)=r_0^2 a(1,x,y)$ and integrating $r_0\in (0,\infty)$, we get 
\begin{align*}
	\frac{32\pi n}{\tau}&\int_0^{\frac{\sqrt 3}{2}} \int_0^{\frac{\sqrt 3}{2}}   \one\{\arcsin(x)\le v_n\}   \left((\pi-3\arcsin(x))+\frac{\sin(3 \arcsin(x))}{\sqrt{1-x^2}}\right)\\
	&\,	\times \one\{\arcsin(y)\le v_n\}  \left((\pi-3\arcsin(y))+\frac{s\sin(3\arcsin(y))}{\sqrt{1-y^2}}\right)\frac{1}{x^3y^3a(1,x,y)^3} \,\mathrm{d}y\,\mathrm{d}x\\
	&+ O(\log n)^2/n.
\end{align*}
It is an elementary calculation to derive that
$$
a(1,x,y)=\frac{\pi-\arcsin(x)}{x^2}+\frac{\sqrt{1-x^2}}{x}+
\frac{\pi-\arcsin(y)}{y^2}+\frac{\sqrt{1-y^2}}{y},\quad  0 < x,y \le \frac{\sqrt 3}{2}.$$
Next we substitute $x=\sin \alpha$ and $y=\sin \beta$ and thus arrive at
\begin{multline*}
	p_n(2) \eq[n]{\infty}	\frac{32\pi n}{\tau} \int_0^{\frac{\pi}{3}\wedge v_n} \int_0^{\frac{\pi}{3}\wedge v_n} \big((\pi-3\alpha)\cos(\alpha)+\sin(3\alpha)\big) \big((\pi-3\beta)\cos(\beta)+\sin(3\beta)\big)\\
	\quad	\times \left(\frac{\sin(\alpha)\sin(\beta)}{\sin(\beta)^2(\pi-\alpha+\sin(\alpha)\cos(\alpha))+ \sin(\alpha)^2(\pi-\beta+\sin(\beta)\cos(\beta))}\right)^3 \,\mathrm{d}\beta\,\mathrm{d}\alpha,
\end{multline*}
and a series expansion of $\sin$ and $\cos$ shows that $p_n(2) =\frac{32 n}{\tau}\int_0^{v_n}\int_0^{v_n}\left( \frac{\alpha\beta}{\alpha^2+\beta^2}  \right)^3\mathrm{d}\beta\mathrm{d}\alpha + O(\log n)^2/n$. This gives $p_n(2)=\frac{2n v_n^2}{\tau}+ O(\log n)^2/n$.  This together with \eqref{eq:defvnangles} implies $p_n(2)=\frac{1}{2}+ O(\log n)^2/n$. 
\end{proof}


\end{document}